%% file: SeifertTAMSFINAL.tex
\documentclass[11pt]{amsart}
\usepackage{amsmath,amssymb,amsthm,graphicx, url, import,verbatim, float,hyperref,ulem} 
\usepackage{enumerate}
\usepackage{xcolor, relsize}

\newtheorem{lemma}{Lemma}[section]
\newtheorem{proposition}[lemma]{Proposition}
\newtheorem{theorem}[lemma]{Theorem}

\newtheorem{corollary}[lemma]{Corollary}
\newtheorem{question}[lemma]{Question}
\newtheorem{conjecture}[lemma]{Conjecture}

\newcommand{\bcon}{\begin{conjecture}}
\newcommand{\econ}{\end{conjecture}}
\newcommand{\bcor}{\begin{corollary}}
\newcommand{\ecor}{\end{corollary}}
\newcommand{\bdf}{\begin{definition}}
\newcommand{\edf}{\end{definition}}
\newcommand{\benu}{\begin{enumerate}}
\newcommand{\eenu}{\end{enumerate}}
\newcommand{\beq}{\begin{equation}}
\newcommand{\eeq}{\end{equation}}
\newcommand{\bexa}{\begin{example}}
\newcommand{\eexa}{\end{example}}
\newcommand{\bexe}{\begin{exercise}}
\newcommand{\eexe}{\end{exercise}}
\newcommand{\bfac}{\begin{fact}}
\newcommand{\efac}{\end{fact}}
\newcommand{\bite}{\begin{itemize}}
\newcommand{\eite}{\end{itemize}}
\newcommand{\blem}{\begin{lemma}}
\newcommand{\elem}{\end{lemma}}
\newcommand{\bmat}{\begin{matrix}}
\newcommand{\emat}{\end{matrix}}
\newcommand{\bprb}{\begin{problem}}
\newcommand{\eprb}{\end{problem}}
\newcommand{\bpro}{\begin{proposition}}
\newcommand{\epro}{\end{proposition}}

\newcommand{\bque}{\begin{question}}
\newcommand{\eque}{\end{question}}
\newcommand{\brem}{\begin{remark}}
\newcommand{\erem}{\end{remark}}
\newcommand{\bthm}{\begin{theorem}}
\newcommand{\ethm}{\end{theorem}}

\newcommand{\bpr}{\begin{proof}}
\newcommand{\epr}{\end{proof}}

\theoremstyle{definition}
\newtheorem{definition}[lemma]{Definition}
\newtheorem{remark}[lemma]{Remark}
\newtheorem{example}[lemma]{Example}

\newtheorem*{namedtheorem}{\theoremname}
\newcommand{\theoremname}{testing}
\newenvironment{named}[1]{\renewcommand{\theoremname}{#1}\begin{namedtheorem}}{\end{namedtheorem}}

\newcommand{\p}{\partial}
\newcommand{\ve}{\varepsilon}

\def\P{\mathbb P}

\newcommand{\Z}{\mathbb{Z}}
\newcommand{\R}{\mathbb{R}}
\renewcommand{\P}{\mathbb{P}}
\newcommand{\Q}{\mathbb{Q}}
\newcommand{\C}{\mathbb{C}}

\newcommand{\K}{\mathbb{K}}

\newcommand{\slC}{\mathrm{SL}_2(\mathbb{C})}
\DeclareMathOperator{\tr}{Tr}
\renewcommand{\S}{\mathcal S}

\title[Skein modules of  Seifert  3-manifolds]{Skein modules and character varieties of Seifert manifolds}

\author{Renaud Detcherry}
\date{} 
\address{Institut de Mathématiques de Bourgogne, UMR 5584 CNRS, Université Bourgogne Franche-Comté, F-2100 Dijon, France}
\email{renaud.detcherry@u-bourgogne.fr}
\author{Efstratia Kalfagianni}
\address{Department of Mathematics, Michigan State University, East
Lansing, MI, 48824, USA}
\email{kalfagia@msu.edu}

\author{Adam S. Sikora}
\address{Department of Mathematics, University at Buffalo,
Buffalo, NY, 14260, USA}
\email{asikora@buffalo.edu}

\def\pmo{{\pm 1}}
\def\cX{\mathcal X}

\begin{document}

\begin{abstract}
We show that the Kauffman bracket skein module of a closed Seifert fibered 3-manifold $M$ is finitely generated over
$\Z[A^{\pm 1}]$ if and only if $M$ is irreducible and non-Haken. 
We analyze in detail the character varieties $\cX(M)$ of such manifolds and show that under mild conditions  they are reduced. We compute the 
 Kauffman bracket skein modules for these $3$-manifolds (over $\Q(A)$) and show that their dimensions coincide with $|\cX(M)|.$
\end{abstract}

\maketitle

\section{Introduction}
\label{sec:intro}

Throughout the paper, $M$ will denote an oriented, closed $3$-manifold.
 Let $\S(M,R)$ be the Kauffman bracket skein module of $M$  with coefficients in a commutative ring $R$, with a distinguished invertible element $A\in R$. The module  $\S(M,R)$ is the quotient of the free $R$-module on  framed unoriented links in $M$, including the empty link $\emptyset$, by the  relations:
\begin{figure}[h]
{\centering
\def \svgwidth{1.1\columnwidth}
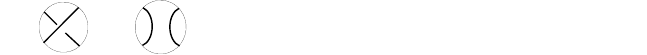}
\end{figure}

 In this paper we will consider the coefficient rings  $R= \Z[A^{\pm 1}], \Q[A^{\pm 1}]$,  and $\Q(A)$.

Recall that a  $\Q[A^{\pm 1}]$-module $S$ is \underline{tame} if it is a direct sum of cyclic $\Q[A^{\pm 1}]$-modules and, for at least one odd $N$,
$S$ does not contain $\Q[A^{\pm 1}]/(\phi_{2N})$ as a submodule,  where $\phi_{2N}$ is the $2N$-th cyclotomic polynomial. In particular, every finitely generated $\Q[A^{\pm 1}]$-module is tame.
We propose the following:

\begin{conjecture}\label{ourconjecture}
For any closed 3-manifold $M$ the following are equivalent:
\begin{enumerate} [(a)]
\item $M$ contains no 2-sided essential surface  
(i.e. $M$ is irreducible, non-Haken). 
\item The skein module $\S(M,\Z[A^{\pm 1}])$ is finitely generated.
\item The skein module $\S(M,\Q[A^{\pm 1}])$ is tame.
\end{enumerate}
\end{conjecture}

All the essential surfaces we discuss in this paper will be 2-sided.

Conjecture \ref{ourconjecture} is inspired and closely related to a conjecture of
Przytycki (Conjecture (E) of \cite[Problem 1.92]{Kirby}) asserting that  $S(M, \Q[A^{\pm 1}])$ is free for manifolds $M$ containing no essential surfaces.
Since $\S(M):=S(M, \Q(A))$  is finitely generated over $\Q(A)$ by
  \cite{GJS19}, Przytycki's conjecture implies that  $S(M, \Q[A^{\pm 1}]))$  is finitely generated
for closed 3-manifold without essential surfaces.

Let
$$\cX(M)=\mathrm{Hom}(\pi_1(M),\slC)/\hspace*{-.05in}/\slC,$$
denote the $\slC$-character variety of $M$, considered as a scheme, as defined for example in \cite{LM85,BH95}.
The coordinate ring  $\C[\cX(M)]$ is the algebra of global sections of the structure sheaf of $\cX(M)$.
The above variety may be non-reduced, that is $\C[\cX(M)]$ may have a nontrivial nil-radical \cite{KM17}.
We denote by $X(M)$ the algebraic set underlying $\cX(M)$ and by $|X(M)|$ the cardinality of
$X(M)$.
By the definitions, 
$\C[X(M)]=\C[\cX(M)]/\sqrt{0}$,
where $\sqrt{0}$ is the nil-radical of $\C[\cX(M)]$.

Our results in  \cite{DKS} imply if $X(M)$ is infinite (for example, then $\S(M,\Q[A^{\pm 1}])$ is non-tame.  On the other hand,
by the Culler-Shalen theory \cite{CullerShalen}, $M$ contains an essential surface. Hence, Conjecture \ref{ourconjecture} holds for 3-manifolds with infinite $X(M)$.
 
One of our main results is:

\begin{theorem}\label{main-i}  
Conjecture \ref{ourconjecture} holds for all Seifert fibered 3-manifolds.
\end{theorem}

An infinite family of non-irreducible  3-manifolds with finite  $X(M)$, is formed by the  connected sum  $M=\R\P^3\# L(2p,1)$, for $p>0$. 
For $M=\R\P^3\# \R\P^3$ the skein module  $\S(M,\Q[A^{\pm 1}])$ was computed in \cite{Mro11} and was shown  to be non-tame.
More recently, Belletti and Detcherry \cite{BD} proved that $\S(M,\Z[A^{\pm 1}])$ is not finitely generated for any $M:=\R\P^3\# L(2p,1)$.
Hence Conjecture \ref{ourconjecture} also holds for this  infinite family of  non-Seifert manifolds as well.

In the course of proving Theorem \ref{main-i}, we also established the following statement, which we have not been able to find in the literature in this generality:

\begin{theorem}\label{t.SeifertX} \label{thm:HakenSFScharac} 
If $M$ is Seifert manifold, then $\dim X(M)>0$ if and only if $M$ is Haken or $M=S^2\times S^1$. 
\end{theorem}

More generally, any $3$-manifold $M$ with infinite $X(M)$ contains an essential surface, by
\cite{CullerShalen}, but the opposite implication is known to be false in general,
as there are Haken or non-irreducible 3-manifolds with finite  $X(M)$.
Besides the connected sums $\R\P^3\# L(n,1)$ considered above, examples of Haken manifolds $M$ with finite
$X(M)$ were constructed for example in \cite{Motegi}.

Combining Theorem \ref{main-i} with an earlier result of the authors \cite[Theorem 1.1]{DKS}, leads to the following corollary proved
 in Section \ref{sec:non-Haken}:
\begin{corollary}\label{reduced} 
For any non-Haken Seifert fibered manifold $M$,
  we have
$$|X(M)| \leq \dim_{\Q(A)}  \S(M) \leq \dim_{\C}\, \C[\cX(M)].$$
In particular, if $\cX(M)$ is reduced then $\dim_{\Q(A)}\S(M)=
|X(M)|$. 
\end{corollary}

 Corollary \ref{reduced} motivates the question of when $\cX(M)$ is  reduced. Non-Haken, irreducible Seifert manifolds
 fiber over $S^2$ with at most three exceptional fibers and non-zero Euler number $e(M)$.
   In Section \ref{sec:characters} we study   $\cX(M)$ of such manifolds  in detail, and we compute the number of their irreducible representations $|X^{irr}(M)|$ (Proposition \ref{prop:irredCharac}). We also show that, under mild conditions on the multiplicities of the exceptional fibers, $\cX(M)$  is reduced.

 We call the  integers $p_1,p_2,p_3$ are \underline{weakly coprime} if one of them is coprime with the other two. 
 
 \begin{theorem}\label{thm:reducedness} Let  $M$ be a Seifert fibered manifold that fibers over $S^2$ with at most three exceptional fibers of multiplicities
 $p_1, p_2, p_3>0$ and $e(M)\neq  0$. If either $H_1(M,\Z)$ is $2$-torsion or if $p_1,p_2,p_3$ are weakly coprime, then $\cX(M)$ is reduced and
$$|X^{irr}(M)|=p_1^+p_2^+p_3^+ + p_1^-p_2^-p_3^-,$$
where $p_i^+=\lceil \frac{p_i}{2}\rceil-1$ and $p_i^-=\lfloor \frac{p_i}{2} \rfloor$.	
\end{theorem}

If  $p_1,p_2,p_3$ are weakly coprime and at most one is even, our count of irreducible characters in Theorem \ref{thm:reducedness} becomes
$$|X^{irr}(M)|=\frac{(p_1-1)(p_2-1)(p_3-1)}{4}.$$
When $p_1,p_2,p_3$  are pairwise coprime,  and furthermore, $q_1,q_2,q_3$ are chosen so that $M$ is a homology sphere, the Seifert manifold $M$ is the Brieskorn homology sphere  $\Sigma(p_1, p_2, p_3)$ and the above formula for
$|X^{irr}(M)|$ was previously discovered in \cite{BC06}, and it gives the $\slC$-Casson invariant of $M$.

Since any rational homology sphere $M$ has $\frac12(|H_1(M,\Z)|+ |H_1(M,\Z/2\Z)|)$ abelian $SL(2,\C)$-characters
(Lemma \ref{lemma:abelianChar}), we conclude:

\begin{corollary}
Under the assumptions of Theorem \ref{thm:reducedness},
$$\dim_{\Q(A)}  \S(M) = p_1^+p_2^+p_3^+ + p_1^-p_2^-p_3^- + \frac12(|H_1(M,\Z)|+ |H_1(M,\Z/2\Z)|).$$
\end{corollary}
For the case of weakly coprime $p_1,p_2,p_3$ we also find explicit bases of $\C[\cX(M)]$ and of $\S(M,\Q(A))$ (Theorem \ref{t:basis}).

It is conjectured  \cite[Section 6.3]{GJS19}
  that $\dim_{\Q(A)} S(M, \Q(A))$ is equal to
the dimension of the zero degree part of the  Abouzaid-Manolescu 
homology $HP^{\bullet}_{\#}(M)$
 \cite{AM20}. By  \cite[Theorem 1.4]{AM20}, if  $\cX(M)$ is finite and reduced, the latter dimension is  $|X(M)|$.  So Theorem \ref{thm:reducedness} and Corollary \ref{reduced} provide new infinite families of 3-manifolds for which the conjecture holds.

\vskip 0.1in

\subsection{Outline of contents}
The paper is organized as follows: In Section \ref{s.skeinmodules} we discuss relations between different, related, properties of skein modules. In Section \ref{sec:Haken} we  prove Theorem \ref{thm:HakenSFScharac}. 
 Section \ref{sec:non-Haken} is devoted to the proof of Theorem \ref{main-i}. First, 
 using our earlier work \cite[Theorem 1.1]{DKS} and Theorem \ref{thm:HakenSFScharac}, 
 we reduce the proof of Theorem \ref{main-i} to showing that the skein module $\S(M,\Z[A^{\pm 1}])$ is finitely generated
 for any Seifert fibered space over $S^2$ with at most three exceptional fibers and with a non-zero Euler number.
 We prove the latter statement by skein-theoretic techniques.
In Section \ref{sec:characters}, we compute  $|X(M)|$ for all  non-Haken Seifert fibered 3-manifolds and, in particular, we prove Theorem \ref{thm:reducedness}.
In the case where $p_1,p_2,p_3$ are weakly coprime we also compute a basis of $\S(M)$.

\vskip 0.1in

\subsection{Acknowledgement}
During the course of this work  Detcherry  was  partially supported by
 the projects ”NAQI-34T” (ANR-23-ERCS-0008) and by the project ``CLICQ" of the R\'egion Bourgogne Franche Comt\'e.
Kalfagianni  was partially supported by the NSF grants DMS-2004155 and DMS-2304033, and Sikora  was partially supported by the Simons Foundation grant 957582.


\section{Relations between different properties of skein modules}\label{s.skeinmodules}
In this section we clarify the relation  of commonly used properties of  3-manifolds skein modules $\S(M,R)$ for  $R=\Q[A^{\pm 1}], \Z[A^{\pm 1}]$,  and $\Q(A)$.
The properties and the corresponding rings are as follows:

\begin{enumerate}
\item $\S(M, R)$ is free over $R=\Z[A^{\pm 1}]$
\item $\S(M, R)$ is finitely generated over $R=\Z[A^{\pm 1}]$.
\item $\S(M,R)$ is tame (over $R=\Q[A^{\pm 1}]$).
\item $\S(M, R)$ is torsion free over $R=\Z[A^{\pm 1}]$.
\item $\S(M)$ has no $(A+1)$- nor $(A-1)$-torsion over $R=\Z[A^{\pm 1}]$.
\item $X(M)$ is finite.
\item $X(M)$ is finite and $\S(M, R)$ is generated over $R=\Z[A^{\pm 1}]$ by $|X(M)|$ elements.
\end{enumerate}

The next proposition describes the relations between above properties.

\begin{proposition} \label{prop:implications} Let $M$ be a closed $3$-manifold. We have the following implications between the properties (1)-(7) above:
$$\begin{array}{rcccccl}

 & &  (2) & \Longrightarrow & (3) & &
\\ & & \Uparrow & &  & &
\\ (7) & \Longrightarrow & (1) & \Rightarrow (4)  \Rightarrow & (5) & \Longrightarrow & (6)	
	
\end{array} $$
Moreover, if $\cX(M)$ is reduced, then we also have $(3)\Rightarrow (5).$
\end{proposition}
\begin{proof}
The implications $(1)\Rightarrow (4) \Rightarrow (5)$ and $(2)\Rightarrow (3)$ are immediate from the definitions.

Suppose that $\S(M, \Z[A^{\pm 1}])$ is free. Then it must have finite rank  since by the finiteness theorem \cite{GJS19} the dimension  $ \dim_{\Q(A)}  \S(M, \Q(A) )$
is finite. Hence we have
the implication $(1) \Rightarrow (2)$.

 The implication $(5)\Rightarrow (6)$ is also a consequence of the Finiteness theorem, though less immediate. A proof is given in  \cite{BD}.
 
The implication $(3)\Rightarrow (6)$ follows from the main theorem of \cite{DKS}. Indeed, if $\S(M)$ is tame, then $\dim_{Q(A)}(\S(M,\Q(A)))\geq |X(M)|,$ therefore $X(M)$ is finite.

For reduced $\cX(M)$, the implication $(3) \Rightarrow (5)$ is the last part of \cite[Theorem 3.1]{DKS}.

Finally, if $\S(M, \Z[A^{\pm 1}])$ is generated by $|X(M)|$ elements, then it is isomorphic to $\Z[A^{\pm 1}]^{|X(M)|}/I$ where $I$ is some $\Z[A^{\pm 1}]$-submodule of $\Z[A^{\pm 1}]^{|X(M)|}.$ If $I$ is not the trivial submodule, then we would have that $S(M,\Q(A))=\left(\Z[A^{\pm 1}]^{|X(M)|}/I\right)\underset{\Z[A^{\pm 1}]}{\otimes}\Q(A)$ would have dimension $<|X(M)|$ over $\Q(A),$ which contradicts \cite[Theorem 3.1]{DKS}. Hence, we must have $I=0,$ and $(7)\Rightarrow (1).$
\end{proof}

\section{Haken Seifert fibered manifolds}
\label{sec:Haken}
Recall that all $3$-manifolds $M$ and essential surfaces in them are assumed closed and oriented in this paper.
A closed, irreducible 3-manifold is \emph{Haken} if it contains an embedded incompressible
surface. The only non-irreducible closed  Seifert fibered 3-manifolds are $S^2\times S^1$ and $\R\P^3\#\R\P^3 $ \cite{Ja}.
In this section, we will prove Theorem \ref{thm:HakenSFScharac}.

We denote by $M(B; \ \frac{q_1}{p_1} \ldots  \frac{q_n}{p_n})$ the Seifert fibered $3$-manifold with the fiber space being the $2$-orbifold $B$ with exceptional fiber invariants $(q_1, p_1)\ldots (q_n, p_n)$ in the notation of  \cite{JN83}.

\begin{proposition} \label{Zetner} Suppose that a closed $3$-manifold $M$ admits a Seifert fibration over $S^2$ with at least four exceptional fibers or over $\R\P^2$ with at least two exceptional fibers. Then $M$ is Haken and $X(M)$ is infinite.
\end{proposition}

We prove Proposition \ref{Zetner} by constructing infinitely many non-conjugate $SU(2)$-representations of $\pi_1(M)$. We will say that a matrix $M\in SU(2)$ has angle $\theta\in [0,\pi]$ if it is conjugated to 
the diagonal matrix with entries $e^{i\theta}$ and $e^{-i\theta}$. 

We will separate the proof of the proposition in two cases, according to whether the base is $S^2$ or $\mathbb{RP}^2.$

We recall that by \cite[Theorem 6.1]{JN83}, the fundamental group of a Seifert manifold of the form $M=M(S^2; \ \frac{q_1}{p_1}, \ldots , \frac{q_n}{p_n})$
has a presentation 
\begin{equation}\label{eq:presentation_pi1}\pi_1(M)=\langle c_1,\ldots,c_n,h\ |\ [h,c_i]=c_i^{p_i} h^{q_i}=1 \ \forall i,\ c_1 \ldots c_n=1 \rangle.
\end{equation}

To treat this case, we will use the following lemma:

\begin{lemma}
	\label{lemma:SZ}\cite[Lemma 2.4]{SZ} There exists a representation
	$$\rho:\langle c_1,c_2,c_3 | c_1c_2c_3 \rangle \longrightarrow SU(2)$$ that maps each $c_i$ to an element of $SU(2)$ of angle $\theta_i$ if and only if
	\begin{equation}\label{eq:condition}|\theta_1-\theta_2|\leq \theta_3 \leq \min(\theta_1+\theta_2,2\pi-\theta_1-\theta_2).
	\end{equation}
\end{lemma}
Note that since the conjugacy class of an element of $SU(2)$ is determined by its angle, the lemma can be alternatively stated as follows: given $\theta_1,\theta_2,\theta_3$ that satisfy Equation \ref{eq:condition}, and $A\in SU(2)$ of angle $\theta_1,$ one can find $B\in SU(2)$ of angle $\theta_2$ so that $AB$ has angle $\theta_3.$

We can now prove the following, which will prove the first case of Proposition \ref{Zetner}:
\begin{lemma}\label{lemma:S2base}
	Let $M=M(S^2; \ \frac{q_1}{p_1}, \ldots , \frac{q_n}{p_n})$ with $2\leq p_1\leq \ldots \leq p_n,$ and let $k\geq 6.$ Then for any angles $\varphi_2,\ldots,\varphi_{n-2} \in [\frac{5\pi}{12},\frac{\pi}{2}],$ there is a representation $\rho:\pi_1(M)\rightarrow SU(2),$ such that $\rho(h)=-I$ and for each $2\leq i \leq n-2,$ the element $\rho(c_1c_2 \ldots c_i)$ has angle $\varphi_i.$
\end{lemma}
\begin{proof}
	Note that the condition $\rho(h)=-I$ implies that the commutation relations in the presentation \ref{eq:presentation_pi1} are automatically satisfied. Moreover, for each $i,$ $\rho(c_i)$ must satisfy $\rho(c_i)^{p_i}=(-1)^{q_i}I.$ As remarked in the proof of \cite[Proposition 2.8]{SZ}, for $p_i\geq 2$ and $(p_i,q_i)$ coprime, $\rho(c_i)$ satisfying the latter condition can always be chosen of angle $\frac{\pi}{4}\leq \theta_i \leq \frac{2\pi}{3}.$

	We will pick the values of $\rho(c_1),\ldots,\rho(c_n)$ inductively so that the angles of the $\rho(c_1c_2 \ldots c_i)$ are always $\varphi_i.$ First, we choose $\rho(c_1)$ and $\rho(c_2)$ so that $\rho(c_1c_2)$ has angle $\varphi_2.$ This is possible since, as we can assume $\frac{\pi}{4}\leq \theta_1,\theta_2 \leq \frac{2\pi}{3},$ the left-hand side of the condition in Lemma \ref{lemma:SZ} is at most $\frac{5\pi}{12},$ and the right hand side at least $\frac{\pi}{2},$ and $\varphi_2 \in [\frac{5\pi}{12},\frac{\pi}{2}].$ 
	
	For the inductive step, assuming we have chosen $\rho(c_1),\ldots,\rho(c_i),$ let us pick $\rho(c_{i+1})$ so that $\rho(c_1\ldots c_{i+1})$ has angle $\varphi_{i+1}.$ Note that the left hand-side of \ref{eq:condition} is less than $\frac{5\pi}{12}$ again, and the right hand side is more than $\frac{\pi}{2},$ so we can pick $\rho(c_{i+1})$ accordingly, by Lemma \ref{lemma:SZ}, since $\varphi_{i+1}\in [\frac{5\pi}{12},\frac{\pi}{2}]$.
	
	Finally, when $\rho(c_1),\ldots ,\rho(c_{n-2})$ have been chosen, we can pick $\rho(c_{n-1}),\rho(c_n)$ by the same reasoning.
	
	The representation $\rho$ now satisfies all relations in the presentation \ref{eq:presentation_pi1}, and furthermore, maps each $c_1\ldots c_i$ for $2\leq i \leq n-2$ to an element of angle $\varphi_i.$ 
\end{proof}

We note that Lemma \ref{lemma:S2base} actually implies that the dimension of $X(M)$ is at least $n-3$ for $M$ a Seifert manifold over $S^2$ with $n\geq 4$ exceptional fibers. We will however not need this fact, but just that $X(M)$ is infinite.
The proof of Proposition \ref{Zetner} now reduces to the other case, manifolds which fiber over $\R\P^2$ with at least $2$ exceptional fibers:

\begin{proof}[Proof of Proposition \ref{Zetner}]

Thanks to Lemma \ref{lemma:S2base}, we only need to treat the case of $M$ which fibers over $\R\P^2$ with at least $2$ exceptional fibers, i.e. $M=M(\R\P^2; \ \frac{q_1}{p_1}, \ldots , \frac{q_n}{p_n})$ with $n\geq 2$ and $2\leq p_1\leq \ldots \leq p_n.$ By \cite[Theorem 6.1]{JN83}, 
$$\pi_1(M)=\langle c_1,\ldots ,c_n,h,a \ \mid \ aha^{-1}=h^{-1},[h,c_i]=c_i^{p_i}h_{q_i}=1 \ \forall i, \ c_1\ldots c_n a^2=1 \rangle$$
Recall that $SU(2)$ can be thought as the unit sphere of the quaternions, with elements 
$$aI +b\begin{pmatrix}
i & 0 \\ 0 & -i
\end{pmatrix}+c\begin{pmatrix}
0 & -1 \\ 1 & 0
\end{pmatrix}+d\begin{pmatrix}
0 & i \\ i & 0
\end{pmatrix}$$
where $a^2+b^2+c^2+d^2=1$ and $I$ is the identity matrix. Unitary quaternions with coordinate $a=0$ along the identity matrix are called purely imaginary.
Following \cite[Proposition 3.4]{SZ}, we define a representation $\rho:\pi_1(M)\longrightarrow SU(2)$ by setting $\rho(h)=-I$ and 
$$\rho(c_k)=\cos\left(\frac{q_k}{p_k}\pi\right) I+ \sin\left(\frac{q_k}{p_k}\pi\right)v_k,$$
for $1\leq k \leq n$ and any purely imaginary unit quaternions $v_1,\ldots,v_n$.
Then taking $\rho(a)$ to be a square root of $\rho(c_1\ldots c_n)^{-1},$ one gets a representation of $\pi_1(M).$ (Note that square roots always exist in $SU(2).$)
	
Now, we notice that 
	$$\mathrm{Tr}(\rho(c_1c_2))=2\cos\left(\frac{q_1}{p_1}\pi\right)\cos\left(\frac{q_2}{p_2}\pi\right)+\sin\left(\frac{q_1}{p_1}\pi\right)\sin\left(\frac{q_2}{p_2}\pi\right)\mathrm{Tr}(v_1v_2).$$
	Since $\frac{q_1}{p_1}$ and $\frac{q_2}{p_2}$ are not integers, $\sin\left(\frac{q_1}{p_1}\pi\right)\sin\left(\frac{q_2}{p_2}\pi\right)\neq 0.$ Consequently,  for different choices of $v_1,v_2$, the trace $\mathrm{Tr}(\rho(c_1c_2))$ can take any value in the interval $[\alpha -\beta, \alpha +\beta]$, for
	$$\alpha:=2\cos\left(\frac{q_1}{p_1}\pi\right)\cos\left(\frac{q_2}{p_2}\pi\right)\ \text{and}\ \beta:=2\left|\sin\left(\frac{q_1}{p_1}\pi\right)\sin\left(\frac{q_2}{p_2}\pi\right)\right|.$$

Since  $\mathrm{Tr}(\rho(c_1c_2))$ can take infinitely many values for $SU(2)$-representations $\rho$ of $\pi_1(M)$, its character variety 
is infinite.
\end{proof} 

Next we prove Theorem \ref{thm:HakenSFScharac} whose statement we recall for the convenience of the reader:

\begin{named}{Theorem \ref{thm:HakenSFScharac}}
If $M$ is Seifert manifold, then $\dim X(M)>0$ if and only if $M$ is Haken or $M=S^2\times S^1$.
\end{named}

\begin{proof}[Proof of Theorem \ref{thm:HakenSFScharac}]
By the work of  Culler and Shallen  \cite{CullerShalen},  if $X(M)$ is infinite, then
$M$ contains an essential surfaces and, hence, it is either Haken or reducible. Since $S^2\times S^1$ and $\R\P^3\#\R\P^3$ are only reducible Seifert manifolds \cite[Lemma VI.7]{Ja} and $X(\R\P^3\#\R\P^3)$ is finite, the implication $\Rightarrow$ follows.

Proof of implication $\Leftarrow:$ Since $X(S^2\times S^1)$ is infinite, we can assume
that $M$ is Haken.
Let $(M, \pi, B)$ be a Seifert fibered space structure on $M$ where $B$ is the orbifold of the fibration
and $\pi: M\longrightarrow B$ is the canonical projection.
Recall that $B$ is a closed surface that may be orientable or non-orientable.
Since $\pi$ induces a surjection of $\pi_1(M)\to \pi_1(B)$, if $B\neq S^2, \R\P^2$, then 
 $H_1(M)$ projects onto $\Z$, implying an infinite $X(M)$ in this case.

Therefore, it is enough to assume that $B$ is either $S^2$ or $\R\P^2$.
In the first case, by Proposition \ref{Zetner}, it is enough to assume that $M$ has at most three exceptional fibers.
Such a manifold is Haken precisely when the fibration has exactly three exceptional fibers and $H_1(M)$ is infinite, see 
\cite[VI.15. Theorem]{Ja}. Thus $X(M)$ is infinite in this case.

If $B= \R\P^2$, then again by Proposition \ref{Zetner}, it is enough to assume that the fibration has
at most one exceptional fiber. Such a 3-manifold $M$ is one of the following \cite[page 97]{Ja}.

\begin{itemize}
\item $M=\R\P^3\#\R\P^3$ 
\item $M$ is a lens space
\item $M$ is a prism manifold, that is a manifold which fibers over $S^2$ with exactly three exceptional fibers of multiplicities $p_1=2,p_2=2,$ and $p_3>1$.
\end{itemize}

In all these cases either $M$ is non-Haken or the result follows from the previous case.
\end{proof}


\section{Non-Haken Seifert  fibered manifolds}\label{sec:non-Haken}
The purpose of this section is to prove Theorem \ref{main-i}, which we recall here: 

\begin{named}{Theorem \ref{main-i}}
For any closed Seifert fibered 3-manifold the following are equivalent:
\begin{enumerate} [(a)] 
\item $M$ contains no 2-sided essential surface (i.e. $M$ is irreducible, non-Haken). 
\item The skein module $\S(M,\Z[A^{\pm 1}])$ is finitely generated.
\item The skein module $\S(M,\Q[A^{\pm 1}])$ is tame.
\end{enumerate}
\end{named}

\begin{proof} (a) $\Rightarrow$ (b): If $M$ contains no 2-sided essential surface then it fibers over $S^2$ with at most three exceptional fibers; that is $M=M(S^2; \ \frac{q_1}{p_1},  \frac{q_2}{p_2}, \frac{q_3}{p_3})$. 
Furthermore, it has finite $H_1(M)$, which happens precisely when the  Euler number, 
$$e(M):= \frac{q_1}{p_1}+ \frac{q_2}{p_2}+ \frac{q_3}{p_3}$$
 is non-zero, \cite{EJa}.

Now the statement follows from Theorem \ref{SFSfingen} below.

The implication (b) $\Rightarrow$ (c) by definition of tameness.

(c) $\Rightarrow$ (a): By \cite[Theorem 1.1]{DKS}, we have $|X(M)| \leq \dim_{\Q(A)}  \S(M)$ and since by \cite{GJS19}, the dimension $\dim_{\Q(A)}  \S(M)$ is finite, we conclude that $X(M)$ is finite. Now by Theorem \ref{thm:HakenSFScharac},
$M$ has to be non-Haken, implying it is either reducible or it contains an incompressible surface.
\end{proof}

\begin{theorem}\label{SFSfingen}  For any Seifert 3-manifold $M$ that fibers over $S^2$ with at most three exceptional fibers and with $e(M)\neq 0$,
the skein module $\S(M,\Z[A^{\pm 1}])$ is finitely generated.
\end{theorem}

The proof of Theorem \ref{SFSfingen} will occupy the remaining of the section.

\subsection{The torus skein algebra and the  Frohman-Gelca  basis}
 
Let $T$ be a torus with a choice of a basis of $H_1(T)$ consisting of simple closed curves $c,h$ oriented to have a single positive intersection.
Let $ \S(T,\Z[A^\pmo])$ denote the skein algebra of the torus $T$
with the  product given by stacking of skeins.
For coprime integers $p,q$, let $(p,q)_T$ denote the simple closed curve on $T$ of slope $p/q$; that is a simple closed curve on $T$ representing $p c+q h$ in $H_1(T)$. More generally, for $p,q\in \Z$, non both zero,  
we set 
$$(p,q)_T=T_d((p/d,q/d))\in \S(T,\Z[A^\pmo]),$$ where $d=gcd(p,q)$ and $T_d(X)$ is the $d$-th Chebyshev polynomial of the first kind.

We use  $(0,0)_T$ to be the empty multicurve, denoted by $\emptyset$.

Since multicurves on $T$ without contractible components form a basis of $\S(T,\Z[A^\pmo])$, the skeins $(p,q)_T$ for $\pm (p,q) \in \Z^2/_{\lbrace \pm 1 \rbrace}$ form a basis of $\S(T,\Z[A^\pmo]).$

We have the product-to-sum formula by \cite{FG00}:
\begin{equation}\label{e.prod}
(p,q)_T\cdot (r, s)_T=A^{ps-qr} (p+r, q+s)_T+ A^{qr-ps} (p-r, q-s)_T.
\end{equation}


\subsection{Realizing Seifert spaces by Dehn fillings on $S_{0,3}\times S^1$}

Let $S_{0,3}$ denote the 3-holed $S^2$, with the boundary components, $c_1,c_2,c_3,$ with their orientation induced by that of $S^2$. Fix a base point $b_i\in c_i,$ for $i=1,2,3$.
Let $N:=S_{0,3} \times S^1$ and
let $T_i=\{c_i\} \times S^1$ and $h_i=\{b_i\} \times S^1$.

Let
$$\ve_1:=q_1/p_1,\quad \ve_2:=q_2/p_2,\quad \ve_3:=q_3/p_3,$$
and let
$M(\ve_1, \ve_2, \ve_3)$ denote the Dehn filling of $N$ along the boundary slopes, $\ve_1, \ve_2, \ve_3$. That is 
$$M=M(S^2; \ \ve_1, \ve_2, \ve_3):= N\cup (V_1\cup V_2\cup V_3),$$ 
where $V_i$ is a solid torus attached to $T_i$ by identifying the meridian of $\partial V_i$ with $p_i c_i+q_i h_i$ in $H_1(T_i)$. The $S^1$-bundle structure of $N$ extents to the  Seifert fibration of $M(S^2; \ \ve_1, \ve_2, \ve_3)$ over $S^2$ with at most three exceptional fibers. The core of $V_i$ is an exceptional fiber iff $p_i>1$, and $p_i$ is the multiplicity of the exceptional fiber in that case.

From now on, let us denote $M(S^2; \ \ve_1, \ve_2, \ve_3)$ by $M(\ \ve_1, \ve_2, \ve_3)$ for simplicity.

Without loss of generality we will assume that $p_1,p_2, p_3\geq 1$ since any closed 3-dimensional Seifert fibered space with fiber space $S^2$ with at most three exceptional fibers (as in Theorem \ref{SFSfingen}) is obtained from $N$ in this way.
We will also assume that the Euler number $e(M)$ is positive, since the orientation reversal of $M$ negates its Euler number and does not affect the statement of Theorem  \ref{SFSfingen}.

Note that
$$\ve_1+\ve_2+\ve_3=e(M),$$ 
and shifting
$\ve_1,\ve_2,\ve_3$ by integers which add up to zero does not change the homeomorphism type of $M( \ve_1, \ve_2, \ve_3)$, cf.
 \cite[Theorem 1.5]{JN83}. Therefore, without loss of generality, we assume that 
\begin{equation}\label{e.ve23}
\ve_2,\ve_3 \leq 0.
\end{equation}
Note that since $e(M)>0,$ this implies that $\ve_1>0,$ and since all $p_i$ are positive, we obtain
\begin{equation}\label{e.q123}q_1>0 \ { \rm {and}} \  q_2,q_2\leq 0.
\end{equation}


\subsection{The skein module  $\S(S_{0,3}\times S^1, \Z[A^{\pm 1}]))$}
. 
The skein module $S(N, \Z[A^{\pm 1}])$
was studied in \cite{Mro-Dab}, where it was shown that it is a free $\Z[A^{\pm 1}]$-module.
 To prove this,
\cite{Mro-Dab} developed a diagrammatic presentation of framed links in products $S\times S^1$, over surfaces $S$. Links are isotoped
to be either  disjoint from $S\times \{1\}$  or to intersect $S\times \{1\}$  transversally and are studied via their projections (diagrams) on $S\times \{1\}$. In this setting,  link diagrams without arrows represent links in a tubular neighborhood of $S\times \{1\}$ in $S\times S^1$, and an arrow
marking on an arc $C$ of a diagram indicates that the arc makes a full loop in $S\times S^1$ in the direction of $S^1$. Changing the direction 
of an arrow on a diagram amounts to changing the direction in which the link runs along the $S^1$ direction.
They show that to study framed links via their diagrams,  besides the three usual  Reidemeister moves on ``unarrowed" parts of diagrams, one needs the following two moves:
\begin{center}
	\def\svgwidth{0.5 \columnwidth}
	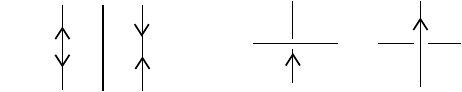
\end{center}

To simplify notation we will represent $S^1$-fibers by ``dots".
  Note that in the Mroczkowski-Dabkowski notation such dots are equal to
\\
\begin{center}
	\def\svgwidth{0.4 \columnwidth}
	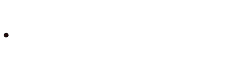
\end{center}
where the two diagrams on the right are arrowed trivial curves on $S$.

In this paper, we will actually not use the result of \cite{Mro-Dab} that $S(N,\Z[A^{\pm 1}])$ is a free over $\Z[A^{\pm 1}]$. Instead, we will be interested in understanding its $S(\partial N,\Z[A^{\pm 1}])$-module structure.
Recall that the skein module $\S(N, \Z[A^{\pm 1}])$ has a natural left module structure over the skein algebra $\S(\partial N, \Z[A^{\pm 1}])$ of the boundary, induced by the homeomorphism $N \coprod_{\partial N} \partial N \times [0,1] \simeq N.$ 

Since the boundary components of $S_{0,3}$ are ordered, the above skein algebra is canonically isomorphic with $\S(T,\Z[A^\pmo]) \otimes \S(T,\Z[A^\pmo]) \otimes\S(T,\Z[A^\pmo]),$  which has a basis
$$\{ (k_1,l_1)_T\otimes (k_2,l_2)_T\otimes (k_3,l_3)_T:\quad   (k_1,l_1),(k_2,l_2),(k_3,l_3)\in \Z^2/\{\pm 1\}\}.$$
We will denote a basis element corresponding to $v=(k_1,l_1,k_2,l_2,k_3,l_3)\in {\Z}^6$ by
$L_v:=L_{k_1,l_1,k_2,l_2,k_3,l_3}$.
Given
$v$ we will denote the skein $L_v\cdot \emptyset\in \S(N, \Z[A^\pmo])$ by 
$\overline{L}_v=\overline{L}_{k_1,l_1,k_2,l_2,k_3,l_3}$, for simplicity.
 
Since all links in $N$ can be isotoped into a tubular neighborhood of $\p N$,  $\S(N, \Z[A^\pmo])$ 
is generated by the set
$$\{\overline{L}_{v} \ | \  v:=(k_1,l_1,k_2,l_2,k_3,l_3)\in {\Z}^6\}.$$

Given a submodule $V\subset \Z^n$ and $w\in \Z^n$, we call the set $w+V$ an \underline{affine subspace} of $\Z^n$ \underline{directed} by $V$.

\begin{lemma}
Suppose that $\S(N, \Z[A^\pmo]))$ is generated by the elements ${\overline L}_{v},$ 
for which $v$ belongs to a finite collection of affine subspaces of $\Z^6$ directed by 
$$V_{p_1,q_1,p_2,q_2,p_3,q_3}=Span_\Z((p_1,q_1,0,0,0,0), (0,0,p_2,q_2,0,0), (0,0,0,0,p_3,q_3)).$$ 
Then $\S(M(\ve_1, \ve_2, \ve_3), \Z[A^\pmo])$ is finitely generated.
\end{lemma} 

\begin{proof} Since the curve of slope $\ve_i=p_i/q_i$ in $T_i$ bounds a disk in $M(\ve_1, \ve_2, \ve_3)$, 
we have

$$((k_1,l_1)_T\cdot  (p_1,q_1)_T, (k_2,l_2)_T, (k_3,l_3)_T)\cdot {\emptyset}=(-A^2-A^{-2}) ( (k_1,l_1)_T, (k_2,l_2)_T, (k_3,l_3)_T)\cdot {\emptyset},$$

or, equivalently,
\begin{equation} \label{er1}
A^{*}\cdot \overline L_{k_1+p_1, l_1+q_1, k_2, l_2, k_3, l_3} +A^{*}\cdot
\overline L_{k_1-p_1, l_1-q_1, k_2, l_2, k_3, l_3} = - (A^2+A^{-2})\cdot \overline L_{k_1, l_1, k_2, l_2, k_3, l_3},
\end{equation}
where $A^{*}$ denotes an unspecified integral power of $A$, by \eqref{e.prod}.

This relation together with the analogous relations corresponding to $T_2$ and $T_3,$ imply that for any 
$w\in \Z^6$ 
the $\Z[A^\pmo]$-submodule of $S(N, \Z[A^\pmo])$ generated by elements of $w+V_{p_1,q_1,p_2,q_2,p_3,q_3}$
is finitely generated. This implies the statement of the proposition.
\end{proof}

For $v:=(k_1,l_1,k_2,l_2,k_3,l_3)$, let 
$$w_i(v):=q_ik_i-p_il_i \ \text{and}\ c_i(v):=|w_i(v)|,$$ 
for $i=1,2,3.$
Let 
$$c(v):=c_1(v)/p_1+c_2(v)/p_2+c_3(v)/p_3$$ 
and let 
$$C(v):=(c(v),-c_1(v))\in \Z^2,$$ 
equipped with lexicographical order.
We call $C(v)$ the \underline{complexity} of $v$.

\begin{proposition}\label{p.reduction1}
For any $v=(k_1,l_1,k_2,l_2,k_3,l_3)$ such that  either $c_2(v)> p_2$ or $c_3(v)> p_3$, 
the element $\overline L_v$ is a linear combination of elements $\overline L_{v'}$, with $C(v')<C(v)$.
\end{proposition}

\begin{proof} 
Let us prove first that $\overline L_v,$ with $c_2(v) >p_2$, can be expressed as $\Z[A^\pmo]$-linear combination of $\overline L_{v'}$, with $v'$ of smaller complexity than $v$. 
The proof for $\overline L_v,$ with $w_3(v) >p_3$, is completely analogous.

Since $(-k,-l)_T=(k,l)_T,$ 
we can assume that $w_i(v)\geq 0$, for $i=1,2,3.$ 
In $\S(T,\Z[A^\pmo])$ we  have 
the following relation obtained by moving the $S^1$ fiber in $N$ so that it lies near each of the three boundary components
of $S_{0,3}$:
\vspace*{.035in} \\
\begin{center}
	\def\svgwidth{0.8 \columnwidth}
	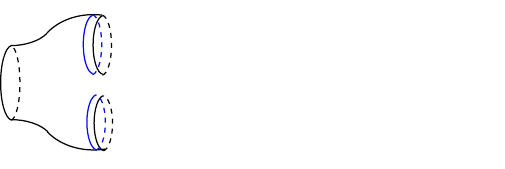
\end{center}
\vspace*{.02in} 
The blue loops are in $S_{0,3}\times \{1\}$ and the dots represent circle fibers.

By the first equality of this relation and \eqref{e.prod}, we get
$$A^{*}\cdot \overline L_{k_1, l_1, k_2, l_2-1, k_3, l_3} + A^{*}\cdot
\overline L_{k_1, l_1, k_2, l_2+1, k_3, l_3} =A^{*}\cdot
\overline L_{k_1, l_1+1, k_2, l_2, k_3, l_3} +
A^{*}\cdot \overline L_{k_1 l_1-1, k_2, l_2, k_3, l_3}.$$
Let $v=(k_1, l_1, k_2, l_2, k_3, l_3)$ and apply above relation to $\overline L_v$.
 By moving the second term to the right and shifting $l_2$ by $1$ we obtain
$$\overline L_v= A^*\cdot \overline L_{k_1, l_1+1, k_2, l_2+1, k_3, l_3}+A^*\cdot
\overline L_{k_1, l_1-1, k_2, l_2+1, k_3, l_3}-A^*\cdot\overline L_{k_1, l_1, k_2, l_2+2, k_3, l_3}.$$

Recall that $p_1, p_2 \geq 1$. Since $w_1,w_2\geq 0$,
$$c(k_1, l_1\pm 1, k_2, l_2+1, k_3, l_3)\leq c(k_1, l_1, k_2, l_2, k_3, l_3),$$
because going from the right to the left side decreases $c_2$ by $p_2$, while it increases $c_1$ by at most $p_1$.

Consequently, the $c$-value, 
$$c(v)=\frac{c_1(v)}{p_1}+ \frac{c_2(v)}{p_2}+ \frac{c_3(v)}{p_3}$$
does not increase.  If the above inequality is actually an equality,  then by the above reasoning 
$$c_1(k_1, l_1\pm 1, k_2, l_2+1, k_3, l_3)> c_1(k_1, l_1, k_2, l_2, k_3, l_3),$$
and, hence, 
$$C(k_1, l_1\pm 1, k_2, l_2+1, k_3, l_3)< C(k_1, l_1, k_2, l_2, k_3, l_3).$$
Finally,
$$C(k_1, l_1, k_2, l_2+2, k_3, l_3)< C(k_1, l_1, k_2, l_2, k_3, l_3),$$
because $c_2(v)=w_2(v)=k_2q_2-l_2p_2>p_2 > 0$ (by the assumption of the proposition) implies $|k_2q_2-(l_2+2)p_2|<|k_2q_2-l_2p_2|$.
\end{proof}

The next ingredient we need for  the proof of Theorem \ref{SFSfingen} is the following:

\begin{proposition}\label{p.reduction2}
Assume that $|\ve_1|+1> \max(|\ve_2-1|+|\ve_3|, |\ve_2|+|\ve_3+1|).$
Then  for any $v$ with $c_1(v)\geq 2|q_1|+2p_1$, $\overline L_v$ can be expressed as a linear combination of elements 
$\overline L_{v'}$, with $C(v')<C(v)$.
\end{proposition}

\begin{proof}
Recall that without loss of generality, we assumed $q_1\geq 0$ (see Equation \ref{e.q123}). Let $v\in \Z^6$ be such that 
\begin{equation}\label{e.q1p1}
c_1(v)\geq 2q_1+2p_1.
\end{equation}
As in the proof of Proposition \ref{p.reduction1}, we will assume that $w_i(v)\geq 0$ for $i=1,2,3,$ without loss of generality. 
We have the following identity of skein elements in $N,$ using the relation $R_5$ introduced above:
\begin{center}
	\def\svgwidth{0.6 \columnwidth}
	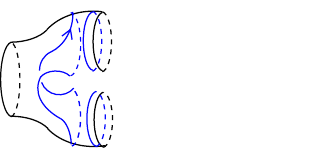
\end{center}
and, by resolving the crossings, we obtain
\vskip 0.03in
\begin{center}
	\def\svgwidth{1 \columnwidth}
	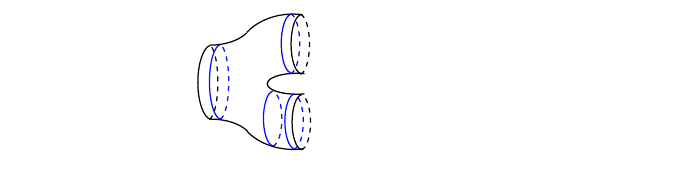
\end{center}

\vskip 0.02in

By applying the product-to-sum formula \eqref{e.prod} we can turn each of the diagrams above into a $A^*$-weighted sum of
$\overline L$-terms. In particular, the first diagram on the left equals $A^* \overline L_{v_1}+A^* \overline L_{v_2},$
where 
$$v_1:= (k_1+1, l_1-1, k_2, l_2, k_3, l_3) \ \text{and}\  v_2:= (k_1-1, l_1+1, k_2, l_2, k_3, l_3),$$
and the second diagram on the left equals

$$A^* \overline L_{v_3}+A^* \overline L_{v_4}+A^* \overline L_{v_5}+A^* \overline L_{v_6},$$
where 
$$v_3:= (k_1, l_1, k_2+1, l_2+1, k_3+1, l_3),\quad v_4:= (k_1, l_1, k_2-1, l_2-1, k_3+1, l_3),$$
and 
$$v_5:= (k_1, l_1, k_2+1, l_2+1, k_3-1, l_3),\quad v_6:= (k_1, l_1, k_2-1, l_2-1, k_3-1, l_3).$$
The $\overline L$-terms on the right are:
$$v_7:=(k_1, l_1, k_2+1, l_2, k_3+1, l_3-1),\quad v_8:=(k_1, l_1, k_2-1, l_2, k_3+1, l_3-1),$$ 
$$v_9:=(k_1, l_1, k_2+1, l_2, k_3-1, l_3+1),\quad v_{10}:=(k_1, l_1, k_2-1, l_2, k_3-1, l_3+1),$$
$$v_{11}:=(k_1+1, l_1+1, k_2, l_2, k_3, l_3),\quad v_{12}:=(k_1-1, l_1-1, k_2, l_2, k_3, l_3).$$

We are going to complete the proof of the proposition
by showing that  $C(v_i)< C(v_1)$, for $2\leq i\leq 12$.
As in the proof of Proposition \ref{p.reduction1}, we assume that $w_1(v_1), w_2(v_1), w_3(v_1)\geq 0$.

Let us analyze vectors $v_2, v_{11},v_{12}$ since they differ from $v_1$ at their first two entries only.
Note that going from $v_1$ to $v_2$ decreases $w_1$ by at most $2q_1+2p_1$ and, consequently, it
decreases the value of $c_1$, by Eq. \eqref{e.q1p1}. Since the values of $c_2$ and $c_3$ remain unchanged,
$C(v_2)<C(v_1).$ 

Going from $v_1$ to $v_{11}$ decreases $w_1$ by at most $2p_1$ and, hence, $C(v_{11})<C(v_1),$ as in the argument above.
Going from $v_1$ to $v_{12}$ decreases $w_1$ by at most $2q_1$ and, hence, $C(v_{11})<C(v_1),$ as well.

Note that the remaining vectors, $v_i$, for $i=3,...,10$, have their first two components $(k_1,l_1).$
Hence, going from $v_1$ to any of them decreases $c_1/p_1$ by $(p_1+q_1)/p_1=1+\ve_1.$
Therefore, by \eqref{e.q1p1}, it is enough to show that going from $v_1$ to one of these vectors increases
$c_2/p_2+c_3/p_3$ by at most $\max(|\ve_2-1|+|\ve_3|, |\ve_2|+|\ve_3+1|)$.

Note that going from $v_1$ to $v_i$ for $i=3,4,5,6,$ increases $c_2/p_2$ by at most $|q_2-p_2|/p_2=|\ve_2-1|$ and it increases $c_3/p_3$ by at most $|q_3|/p_3$. Hence, the above condition holds.

Going $v_1$ to $v_i$ for $i=7,8,9,10,$ increases $c_2/p_2$ by at most $|q_2|/p_2=|\ve_2|$ and it
 it increases $c_3/p_3$ by at most $|q_3+p_3|/p_3=|\ve_3+1|$. Hence, the above condition holds as well.
 \end{proof}
Now we complete the proof of Theorem \ref{SFSfingen} as a corollary of Propositions \ref{p.reduction1} and \ref{p.reduction2}.
\begin{corollary}
If $e(M)=\ve_1+\ve_2+\ve_3\ne 0$, then $\S(M(\ve_1,\ve_2,\ve_3))$ is a finitely generated $\Z[A^\pmo]$-module.
\end{corollary}

\begin{proof}
As before, without loss of generality we assume, that $e(M)>0,$ $\ve_1>0$ and $\ve_2,\ve_3<0$.
Let us first verify the  condition 
$$|\ve_1|+1 > \max(|\ve_2-1|+|\ve_3|, |\ve_2|+|\ve_3+1|),$$
of Proposition \ref{p.reduction2}.
The left side is 
$$\ve_1 +1=e(M)-\ve_2-\ve_3+1> 1-\ve_2-\ve_3,$$ 
while the right side is 
$$\max(-\ve_2+1-\ve_3, -\ve_2 +|\ve_3+1|).$$

Hence, the above inequality follows from the fact that $|\ve_3+1|$ is either $-\ve_3-1$ or is less than $1$.
Recall that $\S(M)$ is generated by the elements $\bar L_v$ for $v\in \Z^6.$ By Propositions \ref{p.reduction1} and \ref{p.reduction2} (which holds by the above inequality), we see that $\S(M)$ is generated by the elements $\bar L_v$ 
satisfying 
$$c_1(v)< 2|q_1|+2p_1,\quad c_2(v)\leq p_2 \ \text{and}\ c_3(v)\leq p_3.$$
There are finitely many of them.
\end{proof}

\vskip 0.02in

Combining Theorem \ref{main-i} and an earlier result of the authors \cite[Theorem 1.1]{DKS}, leads to Corollary \ref{reduced} which we restate here:

\begin{named}{ Corollary \ref{reduced} }
For any non-Haken Seifert fibered manifold $M$,
  we have
$$|X(M)| \leq \dim_{\Q(A)}  \S(M, \Q(A) ) \leq \dim_{\C}\, \C[\cX(M)].$$
In particular, if $\cX(M)$ is reduced then $\dim_{\Q(A)}\S(M)=
|X(M)|$. 
\end{named}

\begin{proof}
It is a direct consequence of Theorem \ref{main-i}, combined with an earlier result of the authors \cite[Theorem 1.1]{DKS}.
Although Theorem \ref{main-i} implies tameness for irreducible $3$-manifolds only, the above statement does hold the reducible Seifert $3$-manifolds, $S^2 \times S^1$ and $\R\P^3\#\R\P^3$. In the first case, $\dim_{\Q(A)} \S(S^2 \times S^1)=1$ by \cite{HP95} and in the latter, 
$$\dim_{\Q(A)} \S(\R\P^3\#\R\P^3)=\dim_{\Q(A)} \S(\R\P^3)\cdot \dim_{\Q(A)} \S(\R\P^3)=4,$$ 
by \cite{HP93} and \cite{Prz00}.
\end{proof}

\section{Character varieties of non-Haken Seifert manifolds}

\label{sec:characters}

In this section we will
 focus on computing  $|X(M)|$, and understanding the extent to which the character scheme  $\cX(M)$ is reduced,
for all  non-Haken Seifert fibered 3-manifolds $M$.  As discussed in the beginning of the proof of Theorem \ref{main-i} such a 3-manifold is either $\R\P^3 \# \R\P^3$
or of the form
$M=M(\frac{q_1}{p_1},  \frac{q_2}{p_2}, \frac{q_3}{p_3})$, where $M$ is a rational homology sphere or (equivalently)
we  have non-zero Euler number (i.e. $\frac{q_1}{p_1}+ \frac{q_2}{p_2}+ \frac{q_3}{p_3}\neq 0$).
We will work with these assumptions throughout the section.
 
We recall that a presentation of $\pi_1(M)$ is 
\begin{equation}\label{eq:pres}
\pi_1(M)=\langle c_1,c_2,c_3,h \ | \ [c_i,h]=1=c_i^{p_i}h^{q_i} \ \text{for}\ i=1,2,3, \ c_1c_2c_3=1 \rangle.
\end{equation}

Let $R(M):= \mathrm{Hom}(\pi_1(M),\slC)$.  A character $\chi$ of $M$
is called abelian if it is the trace of a diagonal representation $\rho \in R(M)$
and central if $\rho$  takes values in the center $\{\pm I\}$ of $SL(2,\C)$.  
A character is called irreducible if it is the trace of an irreducible  $\rho \in R(M)$.

\subsection{Abelian characters}

For any $n\times m$ matrix $P$, where $n\geq m$, let $gcdm(P)$ denote the greatest common divisor of all $m \times m$ minors of $P$. The following lemma will help in counting of abelian characters of $\pi_1(M):$

\begin{lemma}\label{l.gcdm} For any finite abelian group $H$ with an $n\times m$ presentation matrix $P$
$$|Hom(H,\C^*)|=|H|=gcdm(P).$$
\end{lemma}

\begin{proof}
By presenting $H$ as a product of cyclic groups $\Z/k_1 \times ... \times \Z/{k_\ell}$ we see that the number of homomorphisms $H\to \C^*$ is $k_1\cdot ... \cdot {k_\ell}$ and, hence, it coincides with $|H|.$

We say that $P$ is diagonal if its only non-zero entries are of the form $p_{ii}$ for some $1\leq i\leq m.$ If the presentation matrix $P$ of $H$ is diagonal then $|H|=gcdm(P)$. Note now that any $n\times m$ matrix $P$ can be brought to a diagonal one (called its Smith form) by multiplying it by an $n \times n$-matrix on the left and an $m \times m$-matrix on the right. Since these  operations do not change the isomorphism type of $H$  nor 
$gcdm(P)$, it follows that $|H|=gcdm(P)$.
\end{proof}

We begin  with a lemma that computes  the number of abelian characters of $M$.

\begin{lemma} \label{lemma:abelianChar}  (a)\ The number of abelian characters of $M$ is $\frac{1}{2}\left( |H_1(M,\Z)|+|H_1(M,\Z/2\Z)|\right)$.
\vskip 0.02in

(b) \  We have
	$$|H_1(M,\Z)|=|p_1p_2q_3+p_1q_2p_3+q_1p_2p_3|=p_1p_2p_3|e(M)|,$$
	and
	$$ |H_1(M,\Z/2\Z)|=\begin{cases}
		4 \ \ \textrm{if}  \  2 \ \textrm{divides} \ p_1,p_2,p_3,
		\\ 2\  \ \textrm{if} \ 2  \ \textrm{divides}\  p_1p_2p_3|e(M)| \ \textrm{but not all of }p_1,p_2,p_3,
		\\ 1 \  \ \textrm{otherwise.}
	\end{cases}$$
	\end{lemma}
\begin{proof}\ (a) The statement is true for any rational homology sphere:
Up to conjugation, an abelian representation of $\pi_1(M)$ has the form 
	$$\rho_{\lambda}(x)=\begin{pmatrix}
		\lambda(x) & 0 \\ 0 & \lambda(x)^{-1}
	\end{pmatrix},$$
	where $\lambda \in \mathrm{Hom}(\pi_1(M),\C^*)\simeq H_1(M,\Z).$ Two such representations $\rho_{\lambda}$ and $\rho_{\lambda'}$ are conjugated if and only if $\lambda=(\lambda')^{\pm 1}$ (as maps $\pi_1(M)\rightarrow \C^*$). Since one has $\lambda=\lambda^{-1}$ if and only if $\lambda$ takes values in $\lbrace  \pm 1 \rbrace,$ that is, $\lambda \in H_1(M,\Z/2\Z),$ the statement  follows. 
\vskip 0.02in

(b)\ Abelianizing the presentation of $\pi_1(M)$  we get the presentation
$$H_1(M, \Z)=\langle h, c_1,c_2 \ | \   q_i h+p_i c_i= 0, 
\ \text{for}\  i=1,2,\  q_3 h-p_3c_1-p_3 c_2=0\rangle.$$

By Lemma \ref{l.gcdm} the order of $H_1(M,\Z)$ is the absolute value of the determinant of the presentation matrix,
$$\left(\begin{matrix}
q_1 & p_1 & 0  \\ 
	q_2 & 0 & p_2 \\
	q_3 &- p_3 & -p_3
\end{matrix}\right)$$
which gives the first formula, where the matrix is taken  with respect to the ordered set of generators
$\{h, c_1, c_2\}$. 

 For the second formula, note that for each pair $(p_i,q_i)$ at most one of the components is even. So the matrix of above presentation is non-zero mod $2$ and the dimension of $H_1(M,\Z/2\Z)$ over $\Z/2\Z$ is at most $2.$ The dimension is non-zero if and only if the determinant of this matrix is $0$ mod $2.$ For the dimension to be $2$ one needs all of the $2\times 2$ minors: $p_1q_2,\ q_1p_2,\ p_1p_2,\ -q_1p_3-p_1q_3,\ -q_1p_3,\ -p_1p_3,\ -q_2p_3,\ -q_2p_3-p_2q_3,\ p_2p_3$ 
to vanish in $\Z/2\Z$. The later happens if and only if $p_1,p_2$ and $p_3$ are even. 
\end{proof}

\begin{definition}\label{def:excepCharac} An abelian character $\chi$ of $M=M(\frac{q_1}{p_1},  \frac{q_2}{p_2}, \frac{q_3}{p_3})$ is \emph{exceptional} if it is the trace of a representation $\rho \in R(M)$,  such that $\rho(h)=\pm I$ and $\rho(c_i)\neq \pm I$ for $i=1,2,3.$ 
	
	We will use  $x_M$ to denote the number of exceptional abelian characters of $M$.
\end{definition}

\begin{proposition}\label{prop:nbExcepCharac} 
For $\lbrace i,j,k\rbrace=\lbrace 1,2,3\rbrace$, let 
$$m_i:=\gcd(m,4p_j,4p_k,2q_ip_j, 2q_ip_k, 2(p_jq_k+q_jp_k)).$$
We have
$$x_M=\frac{1}{2}(m-m_1-m_2-m_3)+|H_1(M,\Z/2\Z)|,$$
where $m=\gcd (p_1p_2p_3|e(M)|, 2p_1p_2,  2p_1p_3,  2p_2p_3)$.
\end{proposition}

\begin{proof}
	First, note that 
	$$x_M=\frac{1}{2} \left|\lbrace\varphi\in \mathrm{Hom}(\pi_1(M),\C^*) \ | \ \varphi(h)=\pm 1, \varphi(c_i)\neq \pm 1\rbrace \right|,$$
	since any such homomorphism $\phi$ is different from $\phi^{-1},$ and each exceptional character comes from a diagonal representation $\rho_{\phi}$ with diagonal $(\phi,\phi^{-1}),$ and $\rho_{\phi}$ is conjugated to $\rho_{\phi'}$ if and only if $\phi=\phi'^{\pm 1}.$ 
	
By inclusion-exclusion principle, the number of such homomorphisms $\varphi$ is 
$$m-m_1-m_2-m_3+2|H_1(M,\Z/2\Z)|,$$ 
where $m$ is the number of $\varphi\in \mathrm{Hom}(\pi_1(M),\C^*)$ with $\varphi(h)=\pm 1,$ and $m_i$ is the number of $\varphi\in \mathrm{Hom}(\pi_1(M),\C^*)$ such that $\varphi(h)=\varphi(c_i)=\pm 1.$ (Note that if two of $c_1,c_2,c_3$ are mapped to $\pm 1,$ so is the third one).  
By Lemma \ref{l.gcdm}, $m= |H_1(M)/(2h)|$. Eliminating $c_3$ from the presentation of $H_1(M)/(2h)$ that with respect to the ordered set of generators
$\{h, c_1, c_2\}$ the presentation matrix is
$$P=\begin{pmatrix}
	q_1 & p_1 & 0  \\ 
	q_2 & 0 & p_2 \\
	q_3 &- p_3 & -p_3 \\
	2 & 0 & 0 
\end{pmatrix}$$

for which 
\begin{equation}\label{e.gcdmP}
gcdm(P)=\gcd(p_1p_2p_3|e(M)|,2p_1p_2,2p_1p_3,2p_2p_3).
\end{equation} 
	
The computations of $m_1,m_2,m_3$ are similar. For instance,  $m_1=gcdm(P_1)$, where
$$P_1=\begin{pmatrix}
	q_1 & p_1 & 0  \\ q_2 & 0 & p_2 \\ q_3 &  -p_3  & -p_3 \\
	2 & 0 & 0\\  0 & 2 & 0
\end{pmatrix}.$$
Note that  $P_1$ has $\left({5 \atop 2}\right)=10$ of $3 \times 3$-minors, four of which appear in Eq. \eqref{e.gcdmP}.
Taking into account the remaining six we obtain 
\begin{equation}\label{e.m1}
m_1=gcdm(P_1)=\gcd(m,4p_2,4p_3,2q_1p_2,2q_1p_3,2(q_2p_3+p_2q_3)).
\end{equation} 
Since the presentation \eqref{eq:pres}  is preserved by the permutations of indices $1,2,3$,
we obtain formulas for $m_2$ and $m_3$ by permuting indices in \eqref{e.m1}.
The formula for $x_M$ follows.
\end{proof}
\vskip 0.1in

In the remaining of the subsection we give conditions under which we have $x_M=0$ (respectively $x_M>0$). This information is  used in the next subsection to study when
 $\cX(M)$ is reduced. Note that the conclusion of Proposition \ref{p.excepCharac} (respectively, \ref{p.excepCharac2}) holds for all for all  $q_1,q_2,q_3$ coprime with $p_1,p_2,p_3$.
\begin{proposition} 
\label{p.excepCharac}
 If $H_1(M,\Z)$ is $2$-torsion or  if $p_1,p_2, p_3$ are weakly coprime, 
then $x_M=0$.
\end{proposition} 

\begin{proof} If $H_1(M,\Z)$ is $2$-torsion, then $x_M=0$
because all abelian representations have values in $\{\pm I\}$ in this case.

To prove the remaining claim,
assume, without loss of generality, that  $p_1$ is coprime with $p_2, p_3$.
Then $|H_1(M,\Z)|=|p_1p_2q_3+p_1q_2p_3+q_1p_2p_3|$ is coprime with $p_1$, and any $\rho \in R(M)$
such that $\rho(h)=\pm I$,  satisfies $\rho(c_1)^{p_1}=\pm I$.  But since $p_1$ is coprime with $|H_1(M,\Z)|$ we must have $\rho(c_1)=\pm I$ also.	
\end{proof}

On the other hand, we have:

\begin{proposition}\label{p.excepCharac2}Suppose that, for  some $\{i,j,k\}=\{1,2,3\}$,
$d:=gcd(p_i,p_j)>2$ and $s:=gcd(p_ip_j/d,p_k)>2$. If
 either $d\ne 4$ or $s\ne 4$, 
then $x_M>0$ (for all $q_1,q_2,q_3$ coprime with $p_1,p_2,p_3$).
\end{proposition}

\begin{proof}
Without loss of generality assume that $i=1, j=2,k=3.$ 
Since $d\mid \frac{p_1p_2}{s},$ we have
$v_1p_2+v_2p_1= \frac{p_1p_2}{s}$,
for some $v_1,v_2\in \Z$ or, equivalently, 
\begin{equation}\label{e.vp}
\frac{v_1}{p_1}+\frac{v_2}{p_2}= \frac{1}{s}.
\end{equation}
We claim we can choose $v_1,v_2\in \Z$ so that $\frac{v_1}{p_1},\frac{v_2}{p_2}\not\in \frac12\Z$.
Assuming the claim for a moment, note we have an exceptional representation $\rho \in R(M)$, where $\rho(h)=I$ and
$\rho(c_i)= {\rm{Diag} }(e^{2\pi v_i/p_i}, e^{-2\pi v_i/p_i})$, for $i=1,2$. Here, we utilize the fact that $s>2$ and that it divides $p_3$.

Let us prove the above claim now:
Note that integers
$$v_1'=v_1+k \frac{p_1}d,\quad v_2'=v_2- k \frac{p_2}d$$
provide another solution of \eqref{e.vp} for every $k\in\Z$. We argue that $\frac{v'_1}{p_1},\frac{v'_2}{p_2}\not\in \frac12\Z$,
for some $k\in \Z$.
Assume, for a contradiction, that this is not  the case: Then $\frac{v_i}{p_i}\in \frac12\Z$ for at least one $i=1,2$,  say
$\frac{v_1}{p_1}\in \frac12\Z.$ 
We have $\frac{v_1+p_1/d}{p_1}\not\in \frac12\Z$, since $d>2.$ Consequently,
$\frac{v_2-p_2/d}{p_2}\in \frac12\Z$ and, hence, $\frac{v_2-2p_2/d}{p_2}\not\in \frac12\Z$, implying
$\frac{v_1+2p_1/d}{p_1}\in \frac12\Z$. This means that $\frac1d\in \frac14\Z$ implying that $d=4.$ That means that $s\ne 1,2,4$
contradicting \eqref{e.vp}, since $\frac{v_1}{p_1}$ and $\frac{v_2}{p_2}-\frac14$ are in $\frac12\Z$.
\end{proof}
\vskip 0.06in

  It may worth pointing out that if at most one of $p_1,p_2,p_3$ is even,  then the assumptions of Propositions \ref{p.excepCharac} and \ref{p.excepCharac2} are complementary. Specifically, we have:

\begin{corollary}
If at most one of $p_1,p_2,p_3$ is even then $x_M=0$ if and only if $p_1,p_2,p_3$ are weakly coprime.
\end{corollary}

\begin{proof} The implication ($\Rightarrow$) is given by Proposition \ref{p.excepCharac}. To prove the other implication, assume that $p_1,p_2,p_3$ are not weakly coprime. Without loss of generality, we  assume that
$p_1$ is coprime with neither $p_2$ nor $p_3.$ Since at most one of $p_1,p_2,p_3$ is even, $d:=gcd(p_1,p_2)>2$. 
Let $s:=gcd(p_1p_2/d,p_3)$. Since $d$ divides $p_2$, and $p_1,p_3$ are not coprime, we also have $s>1$. And since  
at most one of $p_1,p_2,p_3$ is even, $s>2$ and neither $d$ nor $s$ is $4$. Hence, $x_M\ne 0$ by Proposition \ref{p.excepCharac2}.
\end{proof}

One may show that when more than one $p_i$ is even then the vanishing of $x_M$ depends on $p_1,p_2,p_3$ and on the parities of $q_1,q_2,q_3$ only.

\subsection{Irreducible characters}
\label{ss.irred}
We will denote by $X^{irr}(M)$ the set of irreducible $\slC$-characters of $M.$

\begin{proposition}\label{prop:irredCharac}
	We have $$|X^{irr}(M)|=p_1^+p_2^+p_3^+ + p_1^-p_2^-p_3^- -x_M,$$
	where $p_i^+=\lceil \frac{p_i}{2}\rceil-1$ and $p_i^-=\lfloor \frac{p_i}{2} \rfloor,$ and $x_M$ is as above. 
\end{proposition}

\begin{proof}
Let $X_{2}(M)$, $X_{-2}(M)$ 
denote the set of characters 
$\chi\in X(M)$, where $\chi(c_i)\neq \pm 2$, (for $i=1,2,3$), and
$\chi(h)=2$, $\chi(h)=-2$, respectively.
Since $h$ is central in $\pi_1(M),$ for any irreducible representation $\rho \in R(M)$, we have  $\rho(h)=\pm I$.
Moreover, since $c_i^{p_i}h^{q_i}=1$,  the matrices $\rho(c_i)$ are of finite order. We claim that for an irreducible representation, $\tr (\rho(c_i))\ne \pm 2$. Indeed, if $\tr (\rho(c_i))= \pm 2$, then the finiteness of the order of $\rho(c_i)$ implies that $\rho(c_i)=\pm I$.  Since $\pi_1(M)$ is generated by $h,c_1,c_2$, such $\rho$ must be abelian and thus reducible.

Consequently,
$$X^{irr}(M) = X_2(M)\cup X_{-2}(M) - X_{ea}(M),$$
where $X_{ae}(M)$ is the set of the exceptional abelian characters, and hence
\begin{equation}\label{e.|X^irr|}
|X^{irr}(M)| = |X_2(M)| +|X_{-2}(M)| -x_M.
\end{equation}
For $p\in \Z_{>0}$, set   $$C_p^+:=\{\zeta+\zeta^{-1}\ | \ \zeta^p=1, \zeta\neq \pm 1\},\quad \text{and}\quad 
C_p^-:=\{\zeta+\zeta^{-1}\ |\  \zeta^p=-1, \zeta\neq -1\}.$$ 

By the above discussion, there are functions
$$T_+: X_2(M)\to C_{p_1}^+\times C_{p_2}^+\times C_{p_3}^+$$
$$T_-: X_{-2}(M)\to C_{p_1}^{\epsilon_1} \times C_{p_2}^{\epsilon_2} \times C_{p_3}^{\epsilon_3},$$
with $T_{\pm}(\chi)=(\chi(c_1),\chi(c_2),\chi(c_3)),$ where $\epsilon_i=+$ if $q_i$ is even and  $\epsilon_i=-$ otherwise.
Note that 
$|C_{p_i}^+|=p_i^+= \lceil \frac{p_i}{2}\rceil-1$ and $ |C_p^-|=p_i^-=\lfloor \frac{p_i}{2} \rfloor$.

We will show that $T_{\pm}$ are bijections, which implies
$$|X_2(M)|=p_1^+p_2^+p_3^+\ \text{and}\  |X_{-2}(M)|= p_1^-p_2^-p_3^-,$$
and together with Eq. \eqref{e.|X^irr|} completes the proof.

Each $\chi\in X_{2}(M)$ is either irreducible or exceptional abelian. In either case, $\rho(h)=I$ and $\chi$ is determined by its values
on $F_2=\langle c_1,c_2,c_3 |c_1c_2c_3=1 \rangle$. Since $F_2$ is free of rank two,
$\chi$ is determined by $
(\chi(c_1),\chi(c_2),\chi(c_3))\in \C^3$. Thus $T_+$ is $1$-$1$. Furthermore, each such value in $\C^3$ determines a character of $F_2$. In particular, each value in $C_{p_1}^+\times C_{p_2}^+\times C_{p_3}^+$ corresponds to a character of
$\langle c_1,c_2,c_3 | c_1^{p_1}=c_2^{p_2}=c_3^{p_3}= c_1c_2c_3=1 \rangle$. Thus $T_+$ is onto.
The same argument shows that $T_-$ is a bijection.

\end{proof}

\subsection{Reducedness}
In this subsection we investigate the reduced points of the character varieties  $\cX(M)$. For $\rho \in R(M)$,  let $\mathrm{Ad}\, \rho: \pi_1(M)\to GL(sl_2)$ be the representation of $\pi_1(M)$ on the Lie algebra $sl_2(\C)$ induced by $\rho$ by conjugation.  We recall that 
$$H^1(M,\mathrm{Ad}\, \rho)=Z^1(M,\mathrm{Ad}\, \rho)/B^1(M,\mathrm{Ad}\, \rho),$$ 
where $Z^1(M,\mathrm{Ad}\, \rho)$ is the set of all maps $\varepsilon:\pi_1(M)\longrightarrow sl_2(\C)$ such that:
\begin{equation}\label{cocycle}
 \varepsilon(xy)=\varepsilon(x)+\mathrm{Ad}\, \rho(x) \varepsilon(y), \  \ \text{for all}\ x,y\in \pi_1(M).
\end{equation}	
and $B^1(M,\mathrm{Ad}\, \rho)$ is the subspace of consisting of maps
$$\varepsilon_A(x)=A-\mathrm{Ad}(\rho(x))\cdot A \ \text{for}\ x\in \pi_1(M),$$
for all $A\in sl_2(\C).$

Note that the cocycle condition \eqref{cocycle} implies that
\begin{equation}\label{e.cocycle-n} \varepsilon(x^n)= 
 ( I+\mathrm{Ad}\, \rho(x)+ \ldots + \mathrm{Ad}\, \rho(x^{n-1})) \varepsilon(x)
 \end{equation}
for any $x \in \pi_1(M)$ and for any $n>0$ by induction on $n$.  For the identity element $1\in \pi_1(M)$, since $\varepsilon(1^2)=2\varepsilon(1)$, we get
\begin{equation}\label{e.cocycle-e} \varepsilon(1)=0.
 \end{equation}
 Applying $\varepsilon$ to $x^{-1}\cdot x=1$ implies 
 \begin{equation}\label{e.cocycle-inv} 
 \varepsilon(x^{-1})=-\mathrm{Ad}\, \rho(x^{-1})\varepsilon(x).
  \end{equation}
The rest of the subsection is devoted to proving the following the following:
\begin{theorem}\label{thm:reduced}
	If a Seifert manifold $M=M(\frac{q_1}{p_1}, \frac{q_2}{p_2}, \frac{q_3}{p_3})$ with  $e(M)\neq 0$ has no exceptional abelian characters, then $\cX(M)$ is reduced.
\end{theorem}

We have  the following:

\begin{corollary} If $M$ is a non-Haken Seifert manifold and either $H_1(M,\Z)$ is $2$-torsion or $p_1,p_2,p_3$ are weakly coprime, then $\cX(M)$ is reduced. 
\end{corollary}
\begin{proof} As recalled in the beginning of the section, $M$
 is either $\R\P^3 \# \R\P^3$
or of the form
$M=M(\frac{q_1}{p_1},  \frac{q_2}{p_2}, \frac{q_3}{p_3})$, with $e(M)\neq 0$.
Now the statement follows from Theorem \ref{thm:reduced} and Proposition \ref{p.excepCharac}.
\end{proof}

It is known that each point of $X(M)$ is represented by a completely reducible representation $\rho$ \cite{LM85, Si12}.
In particular, $\rho$ can be taken to be either irreducible or diagonal. Furthermore, the tangent space $T_\rho\, \cX(M)$ at an  irreducible $\rho$ is isomorphic to $H^1(M,\mathrm{Ad}\, \rho)$,  see \cite{LM85}. Similarly, $H^1(M,\mathrm{Ad}\, \rho)=0,$ for a diagonal representation $\rho$ implies that $[\rho]$ is reduced in $\cX(M),$ by \cite[Theorem 1]{Si12} and \cite[Lemma 21]{HeuPor}. Therefore, Theorem \ref{thm:reduced}  follows from Lemmas \ref{lemma:centralReduced},  \ref{lemma:irredReduced}, \ref{lemma:abelianReduced1} and \ref{lemma:excepCharacNonreduced} below.

\begin{lemma}\label{lemma:centralReduced}
	Let $M$ be a rational homology sphere, and let $\chi$ be the character of a central representation in  $ R(M)$. Then $\chi$ is isolated in $X(M)$ and $\cX(M)$ is reduced at $\chi.$
\end{lemma}
\begin{proof}
	Since $\rho$ is a central representation, $\mathrm{Ad}\, \rho$ is a trivial representation, and $H^1(M,\mathrm{Ad}\, \rho)$ is isomorphic to $H^1(M,\C)\otimes sl_2(\C) =0$ since $M$ is a rational homology sphere.
\end{proof}

Next we consider irreducible characters in $\cX(M)$, for  $M=M(\frac{q_1}{p_1}, \frac{q_2}{p_2}, \frac{q_3}{p_3})$.

\begin{lemma}\label{lemma:irredReduced}
	The irreducible  characters are  isolated and reduced in $\cX(M)$.	
\end{lemma}
\begin{proof}
The proof is similar to the argument in \cite[Lemma 2.4]{BC06}:
Since $\rho$ is irreducible, $B^1(M,\mathrm{\rho})$ is of dimension $3,$ so we only need to show that $Z^1(M,\mathrm{Ad}\, \rho)=3$.  We use again the presentation of $\pi_1(M)$ given in Equation  \ref{eq:pres}. Since $h,c_1,c_2,c_3$ generate $\pi_1(M)$, the cocycle condition implies that any $\varepsilon\in Z^1(M,\mathrm{Ad}\, \rho)$ is determined by $H:=\varepsilon(h)$ and $X_i:=\varepsilon(c_i),$ for $i=1,2,3$.
Furthermore,  by \eqref{e.cocycle-inv} 
 we obtain
$\varepsilon(c_i^{-1})= -\mathrm{Ad}\, \rho(c_i^{-1}) \varepsilon(c_i)$.
Hence applying  $\varepsilon$ to the relations $[c_i,h]=1$,
and utilizing properties \eqref{cocycle}, \eqref{e.cocycle-e}, \eqref{e.cocycle-inv}, we obtain
\begin{equation}\label{commute}
	X_i+\mathrm{Ad}\, \rho(c_i)H-\mathrm{Ad}\, \rho(c_ihc_i^{-1})X_i-\mathrm{Ad}\, \rho(c_ihc_i^{-1}h^{-1})H=0.
\end{equation}
Since $\rho$ is irreducible and $\rho(h)$ commutes with $\rho(\pi_1(M))$, we have $\rho(h)=\pm I$ and the last equation reduces to
$$\mathrm{Ad}\, \rho(c_i)H-H=0,$$ 
implying  that $H$ commutes with $\rho(c_i)$ for $i=1,2,3.$ 
Furthermore, since $\rho$ is irreducible, $H$ must be a scalar in $sl_2$ and, hence, 
\begin{equation} \label{e.H0}
H=0.
\end{equation}

By applying $\varepsilon$ to the relation $c_i^{p_i}h^{q_i}=1$, we obtain
$$\varepsilon(c_i^{p_i}) + \mathrm{Ad}\, \rho(c_i^{p_i}) \varepsilon(h^{q_i})=0,$$

which by \eqref{e.cocycle-n} implies
\begin{equation} \label{powers1}
\left(I+\mathrm{Ad}\, \rho(c_i)+\ldots +\mathrm{Ad}\, \rho(c_i^{p_i-1})\right) X_i =-\mathrm{Ad}\, \rho(c_i^{p_i})\left(I+\ldots +\mathrm{Ad}\, \rho(h^{q_i-1})\right) H, 
\end{equation}
and by \eqref{e.H0},
\begin{equation} \label{powers2}
\left(I+\mathrm{Ad}\, \rho(c_i)+\ldots +\mathrm{Ad}\, \rho(c_i^{p_i-1})\right) X_i =0. 
\end{equation}
Since $\rho(c_i)$ is of finite order for each $i$, it is diagonalizable with eigenvalues $\zeta_i,\zeta^{-1}_i$ such that $\zeta_i^{p_i}=(\pm 1)^{q_i}$. 
We note that $\zeta_i\neq \pm 1$ for each $i$. Indeed, if say $\zeta_1=\pm 1$, then the image of $\rho$ is contained in the set 
of powers of $\pm \rho(c_2)$ implying that $\rho$ is abelian.
 In a basis of a diagonalization of $\rho(c_i),$ $\mathrm{Ad}\, \rho(c_i)$ sends  
$X_i=\begin{pmatrix}
	x_i & y_i \\ z_i & -x_i
\end{pmatrix}\in sl_2(\C)$
to 
$\begin{pmatrix}
 x_i & \zeta_i^2y_i  \\ \zeta_i^{-2}z_i & -x_i
\end{pmatrix}$, for $i=1,2,3$.
Since $\zeta_i^{2p_i}=1$, we get
$1+\zeta_i^2+\ldots + \zeta_i^{2p_i-2}=0$, which implies that
that $\left(I+\mathrm{Ad}\, \rho(c_i)+\ldots +\mathrm{Ad}\, \rho(c_i^{p_i-1})\right) X_i$ is diagonal. 
Therefore, 
Equation \eqref{powers2} is satisfied if and only if $X_i$ has diagonal zero, which is equivalent to $\tr(\rho(c_i)X_i)=0$.

Finally, applying $\varepsilon$ to the last relation $c_1c_2c_3=1$, we obtain
\begin{equation} \label{last}
	X_1+\mathrm{Ad}\, \rho(c_1)(X_2)+ \mathrm{Ad}\, \rho(c_1c_2)(X_3)=0.
\end{equation}
We can now describe the space of solutions to the last equation. Given $X_2,X_3\in sl_2(\C)$ such that 
\begin{equation} \label{e.cX}
\tr (\rho(c_2)X_2)=0=\tr (\rho(c_3)X_3),
\end{equation}
  we have
$$X_1=-\mathrm{Ad}\, \rho(c_1)(X_2)- \mathrm{Ad}\, \rho(c_1c_2)(X_3).$$
Note that since $\rho(c_2),\rho(c_3)\neq \pm I$, the space of solutions $(X_2, X_3)$ of Eq. \eqref{e.cX} is $4$-dimensional. 
However, we claim  that the space $Z^1(M,\mathrm{Ad}\, \rho)$ has dimension at most 3  implying $H^1(M,\mathrm{Ad}\, \rho)=0$ as desired.

To show the last claim it is enough to prove that the condition $\tr (\rho(c_1)X_1)=0$ does not hold for all solutions $(X_2, X_3)\in sl_2(\C)^2$ of Eq. \eqref{e.cX}.
Suppose otherwise, for a moment, and take  $X_3=0$.
Then, $X_1=-\mathrm{Ad}\, \rho(c_1)(X_2)$ and $\tr (\rho(c_1)X_1)=0$ becomes $\tr (\rho(c_1)X_2)=0$.
Since the map $M_2(\C)\to M_2(\C)^*$ sending $A$ to $(B\to Tr(AB))$ is injective, the above implies that $\rho(c_1)$ is
a linear combination of $\rho(c_2)$ and the identity matrix. Consequently, $\rho(c_1)$ and $\rho(c_2)$ commute and, hence, $\rho(c_3)=\rho(c_2^{-1}c_1^{-1})$ commutes with them as well. Since $\rho(h)$ commutes with $\rho(c_i)$'s, this  would imply that $\rho$ is abelian, and, hence, not irreducible. This contradiction finishes the proof of the claim and the lemma.
\end{proof}

Next we will  study reducedness of abelian non-central characters. We first look at characters such that $\chi(h)\neq \pm 2.$

\begin{lemma}\label{lemma:abelianReduced1}
If  $\rho \in R(M)$ is diagonal with $\rho(h)\neq \pm I$, then $H^1(M,\mathrm{Ad}\, \rho)=0$ and $\chi_{\rho}$ is reduced in $\cX(M)$.
\end{lemma}

\begin{proof}
We follow the strategy used  in Lemma \ref{lemma:irredReduced}: This time, since $\rho$ is abelian non-central, $\dim B^1(M,\mathrm{Ad}\, \rho)=2$. Let us compute the dimension of $Z^1(M,\mathrm{Ad}\, \rho):$
	
As before, an element $\varepsilon\in Z^1(M,\mathrm{Ad}\, \rho)$ is determined by $H=\varepsilon(h)$ and $X_i=\varepsilon(c_i)$ for $i=1,2,3$. Let	 
$$\rho(h)=\begin{pmatrix} \lambda & 0 \\ 0 & \lambda^{-1} \end{pmatrix}, \ \textrm{and} \ \rho(c_i)=\begin{pmatrix}
\zeta_i & 0 \\ 0 & \zeta_i^{-1} \end{pmatrix}, \ \text{for}\ i=1,2,3,$$
where $\lambda \neq \pm 1,$ and furthermore
$$H=\begin{pmatrix} u & v \\ w & -u \end{pmatrix}, \ \textrm{and} \ X_i=
\begin{pmatrix} x_i & y_i \\ z_i & -x_i \end{pmatrix}.$$
Applying $\varepsilon$ to the equations $[c_i,h]=1$ we obtain \eqref{commute} which now reduces to
$$X_i-\mathrm{Ad}\, \rho(h)X_i=H-\mathrm{Ad}\, \rho(c_i)H.$$
Note that the above matrices have zeros on their diagonals. By comparing the off-diagonal entries, we obtain 
$$y_i=\frac{\zeta_i^2-1}{\lambda^2-1}v, \ z_i=\frac{\zeta_i^{-2}-1}{\lambda^{-2}-1}w.$$

Next applying $\varepsilon$ to the equations $c_i^{p_i}h^{q_i}=1$ we are led again to \eqref{powers1}, which for the diagonal entires reduce to 
\begin{equation} \label{e.xpq}
x_i=-\frac{q_i}{p_i}u.
\end{equation}
Finally, applying $\varepsilon$ to $c_1c_2c_3=1$ once again gives Equation \eqref{last} and 
looking at the diagonal entries, we get 
$x_1+x_2+x_3=0$,
which becomes
\begin{equation} \label{e.sumx}
-\left(\frac{q_1}{p_1}+\frac{q_2}{p_2}+\frac{q_3}{p_3}\right)u=-e(M) u=0.
\end{equation}
Since we assumed that $e(M)\neq 0$, the last equation implies that $u$ vanishes.
Therefore, $\varepsilon$ is entirely determined by $v,w,$ and $\dim Z^1(M,\mathrm{Ad}\, \rho)=2$, implying $\dim H^1(M,\mathrm{Ad}\, \rho)=0.$
\end{proof}

Finally, the last lemma of the subsection treats  the case  of diagonal characters with $\chi(h)=\pm 2$.

\begin{lemma}\label{lemma:excepCharacNonreduced} 
Let $\rho\in R(M)$ be diagonal with $\rho(h)=\pm I$.  Then $H^1(M,\mathrm{Ad}\, \rho)=0$, if $\rho$ is non-exceptional, and $\dim H^1(M,\mathrm{Ad}\, \rho)=2$ otherwise.
\end{lemma}

\begin{proof}
If $\rho$ maps two of  $c_1, c_2, c_3$ to $\pm I$ then so it does the third one and $\rho$ is central. Hence, $\rho$ is  
non-exceptional and the statement follows from Lemma \ref{lemma:centralReduced}. Therefore, we assume without loss of generality, that either $\rho(c_1)=\pm I$ and $\rho(c_2)\neq \pm I$, or  that $\rho$ is exceptional. 

Similarly to the proof of Lemma \ref{lemma:abelianReduced1}, we have $\dim B^1(M,\mathrm{Ad}\, \rho)=2$ and we compute $\dim Z^1(M,\mathrm{Ad}\, \rho)$. 
As before, $\rho(c_i)=\mathrm{Diag}(\zeta_i,\zeta_i^{-1})$, for $i=1,2,3$. We also keep the notations for the entries of $H,X_1,X_2,X_3$.
As in the proof of Lemma \ref{lemma:irredReduced}, applying $\varepsilon$ to
relation $[c_2,h]=1$ gives
$\mathrm{Ad}\, \rho(c_2)(H)-H=0$.
Hence $H$ commutes with $\rho(c_2)$ and thus it is diagonal, i.e. $v=w=0$.
(Now the $1$-cocycle condition is satisfied for $[c_i,h]=1$, $i=1$ and $3$, as well.)

As in the proof of Lemma \ref{lemma:abelianReduced1}, the relations $c_i^{p_i}h^{q_i}=1$, by looking at diagonals imply Eq. \eqref{e.xpq} again:
$$x_i=-\frac{q_i}{p_i}u \ \text{for}\ i=1,2,3.$$

If $\rho(c_i)\ne \pm I$ then \eqref{powers1} is automatically satisfied on off-diagonal entries.
However, 
if $\rho(c_1)=\pm I$ then \eqref{powers1}, and the diagonality of $H$ implies $y_1=z_1=0$. 
Hence, the $1$-cocycle conditions for the relations $[c_i,h]=1$ and $c_i^{p_i}h^{q_i}=1$, define a $7$-dimensional space
of cocycles (determined by parameters: $u, y_i,z_i$ for $i=1,2,3$) for exceptional $\rho$'s and a  $5$-dimensional space
of cocycles (determined by $u, y_i,z_i$ for $i=2,3$) for non-exceptional $\rho$'s.

From the last relation $c_1c_2c_3=1$, we get Eq.  \eqref{e.sumx} again and, hence, $u=0$. Consequently $x_1,x_2,x_3$ vanish as well and by looking at the off-diagonal entires of \eqref{last}, we have
$$y_1+\zeta_1^2y_2+\zeta_1^2\zeta_2^2y_3=0 =  z_1+\zeta_1^2z_2+\zeta_1^2\zeta_2^2z_3.$$ 

If $\rho(c_1)=\pm I,$ then $X_1=0$ by the above discussion and the last equations reduce to
$$y_2+\zeta_2^2y_3=0 = z_2+\zeta_2^{-2}z_3.$$ 
In either case, the above parameters are related by three linearly independent equations stemming from of the relation $c_1c_2c_3=1$.
Now
$Z^1(M,\mathrm{Ad}\, \rho)$ is given by all linear maps $\varepsilon: \pi_1(M)\to sl_2(\C)$ satisfying the $1$-cocycle conditions corresponding to the defining relations of $\pi_1(M)$
(since the $1$-cocycle conditions corresponding to the products of the defining relations are linear combinations of those).
Consequently, $\dim Z^1(M,\mathrm{Ad}\, \rho)$ is either $4$ or $2$-dimensional, depending on whether $\rho$ is exceptional or not. Since $\dim B^1(M,\mathrm{Ad}\, \rho)=2$, the statement follows.
\end{proof}

%
\subsection{Bases for $\C[\cX(M)]$ and $\S(M,\Q(A))$} 
\label{ss.basis}
Let $S_{-1}(M):=S(M, \Z[A^{\pm 1}])\otimes_{\Z[A^{\pm 1}]} \C,$
where the $\Z[A^{\pm 1}]$-module structure of $\C$ is given by sending $A$ to $-1$.
By Przytycki-Sikora \cite{PS00}, $S_{-1}(M)$ has the structure of a $\C$-algebra that is isomorphic to the coordinate ring $\C[\cX(M)]$ of $\cX(M)$, through the isomorphism $\psi: S_{-1}(M)\to \C[\cX(M)]$, sending
any framed link $L=L_1 \cup \ldots \cup L_k$ in $M$ to $(-1)^{k} t_L$ where $t_L=t_{L_1}\cdot \ldots \cdot t_{L_k}$. 
For another approach, up to nilpotents, see \cite{Bullock}. Here, $t_{L_i}$ is the trace function of $L_i$ with its framing ignored.
As we explained  in \cite[Proposition 3.3]{DKS}, if $S(M,\Q[A^{\pm 1}])$ is tame and $\cX(M)$ is reduced, a basis of
 $S_{-1}(M)\simeq \C[\cX(M)]$ leads to a basis of $\S(M)=\S(M, \Q(A))$.

In this subsection we compute a basis of $ \C[\cX(M)]$ for $M=M(\frac{q_1}{p_1}, \frac{q_2}{p_2}, \frac{q_3}{p_3})$.
We will  assume that $p_1,p_2,p_3$ are weakly coprime, since then Proposition \ref{p.excepCharac} implies that there are no exceptional abelian characters and by Theorem \ref{thm:reduced}, the character scheme $\cX(M)$ is reduced.
We will write $y_M$ for the number of abelian characters of $M,$ which was computed in Lemma \ref{lemma:abelianChar}.
With  the notation, $p_i^+=\lceil \frac{p_i}{2}\rceil-1$ and $p_i^--=\lfloor \frac{p_i}{2} \rfloor,$ of Subsection \ref{ss.irred} we have:

 \begin{theorem}\label{t:basis} Assume that $p_1,p_2,p_3$ are weakly coprime, and let $\delta_M\in \lbrace 0,1 \rbrace$ denote the parity of $y_M$.
 Then the following  set is  a basis of $\C[\cX(M)]$ :
	\begin{multline*}\mathcal{B}=\lbrace (t_h+2)t_{c_1}^{k_1}t_{c_2}^{k_2}t_{c_3}^{k_3} \ | \ 0\leq k_i< p_i^+ \rbrace \cup \lbrace (t_h-2) t_{c_1}^{k_1}t_{c_2}^{k_2}t_{c_3}^{k_3} \ | \ 0\leq k_i< p_i^- \rbrace 
		\\ \cup \lbrace t_h^i \ | 2\leq i < y_M + \delta_M\rbrace\cup \mathcal{B}_0,
	\end{multline*}
where
$$\mathcal{B}_0=
\lbrace (t_h+2)t_{c_1}^{p_1^+}\rbrace, \ \text{for}\ y_M \ \text{odd} \ \text{and}\ 
\mathcal{B}_0= \lbrace (t_h+2)t_{c_1}^{p_1^+}, (t_h-2)t_{c_1}^{p_1^-}\rbrace,  \ \text{for}\ y_M \ \text{even}.
$$
Moreover, any collection of $\Z$-linear combinations of links in $M$ that represent those functions in $\C[\cX(M)]\simeq S_{-1}(M)$ is a basis of $S(M,\Q(A)).$
\end{theorem}

\begin{remark} \label{rk:betterbasis} If we assume furthermore that $p_1,p_2,p_3$ are odd, then $p_i^+=p_i^-$, for $i=1,2,3$, and for $y_M$ even one can replace the previous basis by
	$$\mathcal{B}'=\lbrace t_h^i t_{c_1}^{k_1}t_{c_2}^{k_2}t_{c_3}^{k_3} \ | \ 0\leq k_1\leq p_1^+, 0\leq k_2<p_2^+, 0 \leq k_3<p_3^+, 0\leq i\leq 1 \rbrace \cup \lbrace t_h^i \ | 2\leq i < y_M \rbrace,$$
as it can be easily seen that they span the same space. The latter basis consists only of trace functions of links, rather than linear combinations of such functions.
\end{remark}

For the proof of \ref{t:basis} we need the  two preparatory lemmas below:

\begin{lemma}\label{lemma:generatorH_1} If $p_1,p_2,p_3$ are weakly coprime, then $h$ generates $H_1(M,\Z).$ 
\end{lemma}

\begin{proof}
It is enough to show that $H_1(M,\Z)/\langle h \rangle =0$. Using the ordered set of generators
$\{h, c_1, c_2\}$, this group has presentation matrix $Q$ that is identical to the matrix $P$ in the proof of 
Lemma \ref{prop:nbExcepCharac} except that the first entry of the 4th row is 1 instead of 2.
We get $$|H_1(M,\Z)/\langle h \rangle |=gcdm(Q)=\gcd(p_1p_2q_3+p_1p_3q_2+p_2p_3q_1,p_1p_2,p_1p_3,p_2p_3).$$
Without loss of generality assume that $p_1$ is coprime with $p_2$ and with $p_3.$ Then $gcd(p_1p_2,p_1p_3)=p_1$ and, hence, $gcd(p_1p_2,p_1p_3,p_2p_3)=1$ implying that $H_1(M,\Z)/\langle h\rangle$ is trivial.
\end{proof}

The second lemma we need is the following:

\begin{lemma}
	\label{lemma:polynomial}Let $\K$ be a field of characteristic zero and let $P\in \K[X_1,\ldots,X_n]$ be a polynomial of degree $\leq d_i$ in the variable $X_i,$ for $i=1,...,n$. If $P$ vanishes on a subset of $\K^n$ of the form $S_1\times \ldots \times S_n,$ where $|S_i|=d_i+1,$ then $P=0.$
\end{lemma}

\begin{proof}

	We prove the lemma by induction on $n.$ The case $n=1$ is classical. Assume the lemma is true for polynomials in $n$ variables, and let $P\in \K[X_1,\ldots,X_{n+1}]$ satisfy the hypothesis of the lemma. For each $z\in S_{n+1},$ the polynomial $P(X_1,\ldots,X_n,z)$ has degree $\leq d_i$ in each variable $X_i,$ hence it is the zero polynomial. This implies that
$P$ considered in $\K(X_1,...,X_n)[X_{n+1}]$ has at least $d_{n+1}$ roots, implying that $P=0$ by the classical case. 	
\end{proof}

\vskip 0.02in
\noindent { {\bf{Proof of Theorem \ref{t:basis}. }}
The last claim of the theorem follows from the first part by \cite[Proposition 3.3(b)]{DKS}. Let us prove the first part.
	
First we note that $\mathcal{B}$ has the right cardinality by Proposition \ref{prop:irredCharac}. So we only need to show that those trace functions are linearly independent. Since $t_h-2,t_h+2$ belong to first and second subset respectively, one can equivalently replace the third subset by 
$$\lbrace (t_h^2-4)t_h^i \ | \ 0\leq i <y_M-2+\delta_M \rbrace,$$
without affecting the linear independence. We work with the latter version of the third subset.
	
Consider a linear combination $F$
\begin{multline*} F= \sum_{k_1,k_2,k_3} \lambda_{k_1,k_2,k_3} (t_h+2)t_{c_1}^{k_1}t_{c_2}^{k_2}t_{c_3}^{k_3} +
\sum_{k_1,k_2,k_3}  \mu_{k_1,k_2,k_3} (t_h-2)t_{c_1}^{k_1}t_{c_2}^{k_2}t_{c_3}^{k_3}\\
+\sum_j \nu_j (t_h^2-4)t_h^j+a(t_h+2)c_1^{p_1^+}+ b(t_h-2)c_1^{p_1^-},
\end{multline*}
for some coefficients, $\lambda_{k_1,k_2,k_3}, \mu_{k_1,k_2,k_3}, \nu_j,a,b \in \C,$ {(with $b=0$ if $y_M$ is odd)
and assume that it is zero in $\C[\cX(M)]$, i.e. it vanishes on $X(M)$.

Restricting $F$ to the subspace 
$$X_{2,\tau}(M)=\lbrace \chi \in X(M) \ | \ \chi(h)=2, \chi(c_1)=\tau\rbrace$$ of 
$X(M)$
we get:
\begin{equation}\label{e.X_+}
a \tau^{p_1^+}+ \underset{0\leq k_i< p_i^+}{\sum} \lambda_{k_1,k_2,k_3}\tau^{k_1}t_{c_2}^{k_2}t_{c_3}^{k_3}=0 \ \mathrm{on } \  X_{+,\tau}(M),
\end{equation}
since the other components of $F$ vanish on $X_{2,\tau}(M)$. 
Now, recall from the proof of Proposition \ref{prop:irredCharac} that $X_{2}(M)$ 
contains characters taking all possible values $(\chi(c_1),  \chi(c_2),  \chi(c_3)\in C_{p_1}^+\times  C_{p_2}^+\times C_{p_3}^+$.
In other words, characters in $X_{+,\tau}(M)$ take all possible values $(\chi(c_2),  \chi(c_3))$ in $C_{p_2}^+\times C_{p_3}^+$ for every $\tau\in C_{p_1}^+.$
Since for each $\tau\in C_{p_1}^+,$ \eqref{e.X_+} is a polynomial of degree $<p_i^+=|C_{p_i}^+|$, for $i=2,3,$ which vanishes on
$C_{p_2}^+\times C_{p_3}^+$, 
$$a \tau^{p_1^+}+ \underset{0\leq k_1< p_1^+}\sum \lambda_{k_1,k_2,k_3}\tau^{k_1}t_{c_2}^{k_2}t_{c_3}^{k_3}=0 \ \mathrm{on } \  X_{+,\tau}(M),$$
for every $\tau\in C_{p_1}^+,k_2,k_3.$ However, since the above identity also holds for the trivial character and since the above expression is a polynomial in $\tau$ of degree $p_1^+< | C_{p_1}^+|+1=| C_{p_1}^+\cup \{2\}|,$ 
we have $a=\lambda_{k_1,k_2,k_3}=0$ for all $k_1,k_2,k_3$.

 Similarly, if $y_M$ is odd, restricting $F$ to $X_{-2,\tau}(M)=\lbrace \chi \in X(M) \ | \ \chi(h)=-2, \chi(c_1)=\tau \rbrace$ we get $\mu_{k_1,k_2,k_3}=0$ for any $0\leq k_i <p_i^-$, and if $y_M$ is even, we also get that $b=0$, since in the previous argument we can use instead of the trivial character the abelian character such that $\chi(h)=-2$. In particular, $F$ vanishes for $y_M=1.$  
 
 For $y_M>1,$
	$$F=\underset{0\leq j <y_M-2+\delta_M}{\sum} \nu_j (t_h^2-4)t_h^j.$$\
Since $F$ vanishes for $y_M=2$ as well, assume $y_M>2$ now.
Then $G:=\frac{F}{t_h^2-4}$ is a polynomial in $t_h$ of degree $y_M-3+\delta_M$ that vanishes on all of the abelian characters of $H_1(M)$ for which we have $\chi(h)\neq \pm 2$. 
By Lemma \ref{lemma:generatorH_1},  $H_1(M)$ is generated by $h$ and, hence, $t_h$ takes $y_M$ distinct values on abelian characters, including $y_M-2+\delta_M$ values which are not $\pm 2$. This implies that all coefficients of $G$ vanish, i.e. $\nu_j=0$ for every $j.$

\qed

\bibliographystyle{hamsalpha}
\bibliography{biblio}
\end{document}

%% file: kauffman.pdf_tex
\begingroup%
  \makeatletter%
  \providecommand\color[2][]{%
    \errmessage{(Inkscape) Color is used for the text in Inkscape, but the package 'color.sty' is not loaded}%
    \renewcommand\color[2][]{}%
  }%
  \providecommand\transparent[1]{%
    \errmessage{(Inkscape) Transparency is used (non-zero) for the text in Inkscape, but the package 'transparent.sty' is not loaded}%
    \renewcommand\transparent[1]{}%
  }%
  \providecommand\rotatebox[2]{#2}%
  \newcommand*\fsize{\dimexpr\f@size pt\relax}%
  \newcommand*\lineheight[1]{\fontsize{\fsize}{#1\fsize}\selectfont}%
  \ifx\svgwidth\undefined%
    \setlength{\unitlength}{310.27629089bp}%
    \ifx\svgscale\undefined%
      \relax%
    \else%
      \setlength{\unitlength}{\unitlength * \real{\svgscale}}%
    \fi%
  \else%
    \setlength{\unitlength}{\svgwidth}%
  \fi%
  \global\let\svgwidth\undefined%
  \global\let\svgscale\undefined%
  \makeatother%
  \begin{picture}(1,0.08365284)%
    \lineheight{1}%
    \setlength\tabcolsep{0pt}%
    \put(0,-.01){\includegraphics[width=\unitlength,page=1]{kauffman.pdf}}%
    \put(0.1543921,0.02370045){\color[rgb]{0,0,0}\makebox(0,0)[lt]{\lineheight{0}\smash{\begin{tabular}[t]{l}$=A$\end{tabular}}}}%
    \put(0.29591666,0.02367871){\color[rgb]{0,0,0}\makebox(0,0)[lt]{\lineheight{0}\smash{\begin{tabular}[t]{l}$+A^{-1}$\end{tabular}}}}%
    \put(0.57577529,-0.03838106){\color[rgb]{0,0,0}\makebox(0,0)[lt]{\lineheight{0}\smash{\begin{tabular}[t]{l} \end{tabular}}}}%
    \put(0.53325876,0.02471831){\color[rgb]{0,0,0}\makebox(0,0)[lt]{\lineheight{0}\smash{\begin{tabular}[t]{l}$L \ \sqcup$ \end{tabular}}}}%
    \put(0,-.01){\includegraphics[width=\unitlength,page=2]{kauffman.pdf}}%
    \put(0.67966381,0.02531714){\color[rgb]{0,0,0}\makebox(0,0)[lt]{\lineheight{0}\smash{\begin{tabular}[t]{l}$=(-A^2-A^{-2}) L$\end{tabular}}}}%
    \put(-0.00304793,0.0192645){\color[rgb]{0,0,0}\makebox(0,0)[lt]{\lineheight{0}\smash{\begin{tabular}[t]{l}K1:\end{tabular}}}}%
    \put(0.47835686,0.02531722){\color[rgb]{0,0,0}\makebox(0,0)[lt]{\lineheight{0}\smash{\begin{tabular}[t]{l}K2:\end{tabular}}}}%
    \put(0,-.01){\includegraphics[width=\unitlength,page=3]{kauffman.pdf}}%
  \end{picture}%
\endgroup%

%% file: Reidemeister.pdf_tex
\begingroup%
  \makeatletter%
  \providecommand\color[2][]{%
    \errmessage{(Inkscape) Color is used for the text in Inkscape, but the package 'color.sty' is not loaded}%
    \renewcommand\color[2][]{}%
  }%
  \providecommand\transparent[1]{%
    \errmessage{(Inkscape) Transparency is used (non-zero) for the text in Inkscape, but the package 'transparent.sty' is not loaded}%
    \renewcommand\transparent[1]{}%
  }%
  \providecommand\rotatebox[2]{#2}%
  \ifx\svgwidth\undefined%
    \setlength{\unitlength}{221.64836426bp}%
    \ifx\svgscale\undefined%
      \relax%
    \else%
      \setlength{\unitlength}{\unitlength * \real{\svgscale}}%
    \fi%
  \else%
    \setlength{\unitlength}{\svgwidth}%
  \fi%
  \global\let\svgwidth\undefined%
  \global\let\svgscale\undefined%
  \makeatother%
  \begin{picture}(1,0.1979595)%
    \put(0,0){\includegraphics[width=\unitlength]{Reidemeister.pdf}}%
    \put(0.1511898,0.08198808){\color[rgb]{0,0,0}\makebox(0,0)[lb]{\smash{$\sim$}}}%
    \put(0.23048952,0.08198808){\color[rgb]{0,0,0}\makebox(0,0)[lb]{\smash{$\sim$}}}%
    \put(-0.00102922,0.0792536){\color[rgb]{0,0,0}\makebox(0,0)[lb]{\smash{$(R_4)$}}}%
    \put(0.43557502,0.0792536){\color[rgb]{0,0,0}\makebox(0,0)[lb]{\smash{$(R_5)$}}}%
    \put(0.74001303,0.07469615){\color[rgb]{0,0,0}\makebox(0,0)[lb]{\smash{$\sim$}}}%
  \end{picture}%
\endgroup%

%% file: FiberNotation.pdf_tex
\begingroup%
  \makeatletter%
  \providecommand\color[2][]{%
    \errmessage{(Inkscape) Color is used for the text in Inkscape, but the package 'color.sty' is not loaded}%
    \renewcommand\color[2][]{}%
  }%
  \providecommand\transparent[1]{%
    \errmessage{(Inkscape) Transparency is used (non-zero) for the text in Inkscape, but the package 'transparent.sty' is not loaded}%
    \renewcommand\transparent[1]{}%
  }%
  \providecommand\rotatebox[2]{#2}%
  \newcommand*\fsize{\dimexpr\f@size pt\relax}%
  \newcommand*\lineheight[1]{\fontsize{\fsize}{#1\fsize}\selectfont}%
  \ifx\svgwidth\undefined%
    \setlength{\unitlength}{110.95725352bp}%
    \ifx\svgscale\undefined%
      \relax%
    \else%
      \setlength{\unitlength}{\unitlength * \real{\svgscale}}%
    \fi%
  \else%
    \setlength{\unitlength}{\svgwidth}%
  \fi%
  \global\let\svgwidth\undefined%
  \global\let\svgscale\undefined%
  \makeatother%
  \begin{picture}(1,0.26880558)%
    \lineheight{1}%
    \setlength\tabcolsep{0pt}%
    \put(0,0){\includegraphics[width=\unitlength,page=1]{FiberNotation.pdf}}%
    \put(0.0835308,0.10736947){\color[rgb]{0.16862745,0,0}\makebox(0,0)[lt]{\lineheight{1.25}\smash{\begin{tabular}[t]{l}$=-A^3$\end{tabular}}}}%
    \put(0,0){\includegraphics[width=\unitlength,page=2]{FiberNotation.pdf}}%
    \put(0.55523564,0.1058672){\color[rgb]{0.16862745,0,0}\makebox(0,0)[lt]{\lineheight{1.25}\smash{\begin{tabular}[t]{l}$=-A^{-3}$\end{tabular}}}}%
    \put(0,0){\includegraphics[width=\unitlength,page=3]{FiberNotation.pdf}}%
  \end{picture}%
\endgroup%

%% file: Relation1.pdf_tex
\begingroup%
  \makeatletter%
  \providecommand\color[2][]{%
    \errmessage{(Inkscape) Color is used for the text in Inkscape, but the package 'color.sty' is not loaded}%
    \renewcommand\color[2][]{}%
  }%
  \providecommand\transparent[1]{%
    \errmessage{(Inkscape) Transparency is used (non-zero) for the text in Inkscape, but the package 'transparent.sty' is not loaded}%
    \renewcommand\transparent[1]{}%
  }%
  \providecommand\rotatebox[2]{#2}%
  \newcommand*\fsize{\dimexpr\f@size pt\relax}%
  \newcommand*\lineheight[1]{\fontsize{\fsize}{#1\fsize}\selectfont}%
  \ifx\svgwidth\undefined%
    \setlength{\unitlength}{242.06896925bp}%
    \ifx\svgscale\undefined%
      \relax%
    \else%
      \setlength{\unitlength}{\unitlength * \real{\svgscale}}%
    \fi%
  \else%
    \setlength{\unitlength}{\svgwidth}%
  \fi%
  \global\let\svgwidth\undefined%
  \global\let\svgscale\undefined%
  \makeatother%
  \begin{picture}(1,0.33482259)%
    \lineheight{1}%
    \setlength\tabcolsep{0pt}%
    \put(0,0){\includegraphics[width=\unitlength,page=1]{Relation1.pdf}}%
    \put(0.26077451,0.14291762){\color[rgb]{0,0,0}\makebox(0,0)[lt]{\lineheight{0}\smash{\begin{tabular}[t]{l}$=$\end{tabular}}}}%
    \put(0,0){\includegraphics[width=\unitlength,page=2]{Relation1.pdf}}%
    \put(-0.01376147,0.27911541){\color[rgb]{0,0,1}\makebox(0,0)[lt]{\lineheight{1.25}\smash{\begin{tabular}[t]{l}$(k_1,l_1)_T$\end{tabular}}}}%
    \put(0.16925388,0.31547317){\color[rgb]{0,0,1}\makebox(0,0)[lt]{\lineheight{1.25}\smash{\begin{tabular}[t]{l}$(k_2,l_2)_T$\end{tabular}}}}%
    \put(0.17729287,0.00436209){\color[rgb]{0,0,1}\makebox(0,0)[lt]{\lineheight{1.25}\smash{\begin{tabular}[t]{l}$(k_3,l_3)_T$\end{tabular}}}}%
    \put(0,0){\includegraphics[width=\unitlength,page=3]{Relation1.pdf}}%
    \put(0.32307152,0.26145859){\color[rgb]{0,0,1}\makebox(0,0)[lt]{\lineheight{1.25}\smash{\begin{tabular}[t]{l}$(k_1,l_1)_T$\end{tabular}}}}%
    \put(0.52655323,0.3133708){\color[rgb]{0,0,1}\makebox(0,0)[lt]{\lineheight{1.25}\smash{\begin{tabular}[t]{l}$(k_2,l_2)_T$\end{tabular}}}}%
    \put(0.53377352,0.00389707){\color[rgb]{0,0,1}\makebox(0,0)[lt]{\lineheight{1.25}\smash{\begin{tabular}[t]{l}$(k_3,l_3)_T$\end{tabular}}}}%
    \put(0,0){\includegraphics[width=\unitlength,page=4]{Relation1.pdf}}%
    \put(0.06813857,0.18908224){\color[rgb]{0,0,1}\makebox(0,0)[lt]{\lineheight{1.25}\smash{\begin{tabular}[t]{l}$(0,1)_T$\end{tabular}}}}%
    \put(0,0){\includegraphics[width=\unitlength,page=5]{Relation1.pdf}}%
    \put(0.43938401,0.19030723){\color[rgb]{0,0,1}\makebox(0,0)[lt]{\lineheight{1.25}\smash{\begin{tabular}[t]{l}$(0,1)_T$\end{tabular}}}}%
    \put(0,0){\includegraphics[width=\unitlength,page=6]{Relation1.pdf}}%
    \put(0.67622539,0.27526777){\color[rgb]{0,0,1}\makebox(0,0)[lt]{\lineheight{1.25}\smash{\begin{tabular}[t]{l}$(k_1,l_1)_T$\end{tabular}}}}%
    \put(0.87806979,0.32226807){\color[rgb]{0,0,1}\makebox(0,0)[lt]{\lineheight{1.25}\smash{\begin{tabular}[t]{l}$(k_2,l_2)_T$\end{tabular}}}}%
    \put(0.88365278,0.00542641){\color[rgb]{0,0,1}\makebox(0,0)[lt]{\lineheight{1.25}\smash{\begin{tabular}[t]{l}$(k_3,l_3)_T$\end{tabular}}}}%
    \put(0,0){\includegraphics[width=\unitlength,page=7]{Relation1.pdf}}%
    \put(0.79253797,0.09236999){\color[rgb]{0,0,1}\makebox(0,0)[lt]{\lineheight{1.25}\smash{\begin{tabular}[t]{l}$(0,1)_T$\end{tabular}}}}%
    \put(0.60874755,0.14790847){\color[rgb]{0,0,0}\makebox(0,0)[lt]{\lineheight{0}\smash{\begin{tabular}[t]{l}$=$\end{tabular}}}}%
  \end{picture}%
\endgroup%

%% file: Relation2_proof.pdf_tex
\begingroup%
  \makeatletter%
  \providecommand\color[2][]{%
    \errmessage{(Inkscape) Color is used for the text in Inkscape, but the package 'color.sty' is not loaded}%
    \renewcommand\color[2][]{}%
  }%
  \providecommand\transparent[1]{%
    \errmessage{(Inkscape) Transparency is used (non-zero) for the text in Inkscape, but the package 'transparent.sty' is not loaded}%
    \renewcommand\transparent[1]{}%
  }%
  \providecommand\rotatebox[2]{#2}%
  \newcommand*\fsize{\dimexpr\f@size pt\relax}%
  \newcommand*\lineheight[1]{\fontsize{\fsize}{#1\fsize}\selectfont}%
  \ifx\svgwidth\undefined%
    \setlength{\unitlength}{155.314991bp}%
    \ifx\svgscale\undefined%
      \relax%
    \else%
      \setlength{\unitlength}{\unitlength * \real{\svgscale}}%
    \fi%
  \else%
    \setlength{\unitlength}{\svgwidth}%
  \fi%
  \global\let\svgwidth\undefined%
  \global\let\svgscale\undefined%
  \makeatother%
  \begin{picture}(1,0.50232194)%
    \lineheight{1}%
    \setlength\tabcolsep{0pt}%
    \put(0,0){\includegraphics[width=\unitlength,page=1]{Relation2_proof.pdf}}%
    \put(0.4063941,0.21381531){\color[rgb]{0,0,0}\makebox(0,0)[lt]{\lineheight{0}\smash{\begin{tabular}[t]{l}$=$\end{tabular}}}}%
    \put(0,0){\includegraphics[width=\unitlength,page=2]{Relation2_proof.pdf}}%
    \put(0.0008189,0.42299003){\color[rgb]{0,0,1}\makebox(0,0)[lt]{\lineheight{1.25}\smash{\begin{tabular}[t]{l}$(k_1,l_1)_T$\end{tabular}}}}%
    \put(0.26375295,0.48275484){\color[rgb]{0,0,1}\makebox(0,0)[lt]{\lineheight{1.25}\smash{\begin{tabular}[t]{l}$(k_2,l_2)_T$\end{tabular}}}}%
    \put(0.27500635,0.0093505){\color[rgb]{0,0,1}\makebox(0,0)[lt]{\lineheight{1.25}\smash{\begin{tabular}[t]{l}$(k_3,l_3)_T$\end{tabular}}}}%
    \put(0,0){\includegraphics[width=\unitlength,page=3]{Relation2_proof.pdf}}%
    \put(0.49733299,0.40591654){\color[rgb]{0,0,1}\makebox(0,0)[lt]{\lineheight{1.25}\smash{\begin{tabular}[t]{l}$(k_1,l_1)_T$\end{tabular}}}}%
    \put(0.82062825,0.47947815){\color[rgb]{0,0,1}\makebox(0,0)[lt]{\lineheight{1.25}\smash{\begin{tabular}[t]{l}$(k_2,l_2)_T$\end{tabular}}}}%
    \put(0.83188158,0.00607381){\color[rgb]{0,0,1}\makebox(0,0)[lt]{\lineheight{1.25}\smash{\begin{tabular}[t]{l}$(k_3,l_3)_T$\end{tabular}}}}%
  \end{picture}%
\endgroup%

%% file: Relation2.pdf_tex
\begingroup%
  \makeatletter%
  \providecommand\color[2][]{%
    \errmessage{(Inkscape) Color is used for the text in Inkscape, but the package 'color.sty' is not loaded}%
    \renewcommand\color[2][]{}%
  }%
  \providecommand\transparent[1]{%
    \errmessage{(Inkscape) Transparency is used (non-zero) for the text in Inkscape, but the package 'transparent.sty' is not loaded}%
    \renewcommand\transparent[1]{}%
  }%
  \providecommand\rotatebox[2]{#2}%
  \newcommand*\fsize{\dimexpr\f@size pt\relax}%
  \newcommand*\lineheight[1]{\fontsize{\fsize}{#1\fsize}\selectfont}%
  \ifx\svgwidth\undefined%
    \setlength{\unitlength}{328.67505892bp}%
    \ifx\svgscale\undefined%
      \relax%
    \else%
      \setlength{\unitlength}{\unitlength * \real{\svgscale}}%
    \fi%
  \else%
    \setlength{\unitlength}{\svgwidth}%
  \fi%
  \global\let\svgwidth\undefined%
  \global\let\svgscale\undefined%
  \makeatother%
  \begin{picture}(1,0.24545472)%
    \lineheight{1}%
    \setlength\tabcolsep{0pt}%
    \put(0.47638891,0.11563234){\color[rgb]{0,0,0}\makebox(0,0)[lt]{\lineheight{0}\smash{\begin{tabular}[t]{l}$=A$\end{tabular}}}}%
    \put(0,0){\includegraphics[width=\unitlength,page=1]{Relation2.pdf}}%
    \put(0.24641706,0.1997338){\color[rgb]{0,0,1}\makebox(0,0)[lt]{\lineheight{1.25}\smash{\begin{tabular}[t]{l}$(k_1,l_1)_T$\end{tabular}}}}%
    \put(0.42944283,0.23254067){\color[rgb]{0,0,1}\makebox(0,0)[lt]{\lineheight{1.25}\smash{\begin{tabular}[t]{l}$(k_2,l_2)_T$\end{tabular}}}}%
    \put(0.43656947,0.00943694){\color[rgb]{0,0,1}\makebox(0,0)[lt]{\lineheight{1.25}\smash{\begin{tabular}[t]{l}$(k_3,l_3)_T$\end{tabular}}}}%
    \put(0,0){\includegraphics[width=\unitlength,page=2]{Relation2.pdf}}%
    \put(-0.00016644,0.19872473){\color[rgb]{0,0,1}\makebox(0,0)[lt]{\lineheight{1.25}\smash{\begin{tabular}[t]{l}$(k_1,l_1)_T$\end{tabular}}}}%
    \put(0.15452125,0.23454628){\color[rgb]{0,0,1}\makebox(0,0)[lt]{\lineheight{1.25}\smash{\begin{tabular}[t]{l}$(k_2,l_2)_T$\end{tabular}}}}%
    \put(0.15923609,0.00903082){\color[rgb]{0,0,1}\makebox(0,0)[lt]{\lineheight{1.25}\smash{\begin{tabular}[t]{l}$(k_3,l_3)_T$\end{tabular}}}}%
    \put(0,0){\includegraphics[width=\unitlength,page=3]{Relation2.pdf}}%
    \put(0.05892747,0.02042867){\color[rgb]{0,0,1}\makebox(0,0)[lt]{\lineheight{1.25}\smash{\begin{tabular}[t]{l}$(1,-1)_T$\end{tabular}}}}%
    \put(-0.00101021,0.1301376){\color[rgb]{0.15686275,0.04313725,0.04313725}\makebox(0,0)[lt]{\lineheight{1.25}\smash{\begin{tabular}[t]{l}$A$\end{tabular}}}}%
    \put(0.19894438,0.1271532){\color[rgb]{0.15686275,0.04313725,0.04313725}\makebox(0,0)[lt]{\lineheight{1.25}\smash{\begin{tabular}[t]{l}$+A^{-1}$\end{tabular}}}}%
    \put(0,0){\includegraphics[width=\unitlength,page=4]{Relation2.pdf}}%
    \put(0.75704998,0.1961817){\color[rgb]{0,0,1}\makebox(0,0)[lt]{\lineheight{1.25}\smash{\begin{tabular}[t]{l}$(k_1,l_1)_T$\end{tabular}}}}%
    \put(0.91897296,0.23320914){\color[rgb]{0,0,1}\makebox(0,0)[lt]{\lineheight{1.25}\smash{\begin{tabular}[t]{l}$(k_2,l_2)_T$\end{tabular}}}}%
    \put(0.92368788,0.00889962){\color[rgb]{0,0,1}\makebox(0,0)[lt]{\lineheight{1.25}\smash{\begin{tabular}[t]{l}$(k_3,l_3)_T$\end{tabular}}}}%
    \put(0,0){\includegraphics[width=\unitlength,page=5]{Relation2.pdf}}%
    \put(0.82433536,0.02062027){\color[rgb]{0,0,1}\makebox(0,0)[lt]{\lineheight{1.25}\smash{\begin{tabular}[t]{l}$(1,1)_T$\end{tabular}}}}%
    \put(0,0){\includegraphics[width=\unitlength,page=6]{Relation2.pdf}}%
    \put(0.34836106,0.23057784){\color[rgb]{0,0,1}\makebox(0,0)[lt]{\lineheight{1.25}\smash{\begin{tabular}[t]{l}$(1,1)_T$\end{tabular}}}}%
    \put(0.34865377,0.00980561){\color[rgb]{0,0,1}\makebox(0,0)[lt]{\lineheight{1.25}\smash{\begin{tabular}[t]{l}$(1,0)_T$\end{tabular}}}}%
    \put(0.70299194,0.11918647){\color[rgb]{0.15686275,0.04313725,0.04313725}\makebox(0,0)[lt]{\lineheight{1.25}\smash{\begin{tabular}[t]{l}$+A^{-1}$\end{tabular}}}}%
    \put(0,0){\includegraphics[width=\unitlength,page=7]{Relation2.pdf}}%
    \put(0.50008069,0.20038676){\color[rgb]{0,0,1}\makebox(0,0)[lt]{\lineheight{1.25}\smash{\begin{tabular}[t]{l}$(k_1,l_1)_T$\end{tabular}}}}%
    \put(0.66682699,0.23439952){\color[rgb]{0,0,1}\makebox(0,0)[lt]{\lineheight{1.25}\smash{\begin{tabular}[t]{l}$(k_2,l_2)_T$\end{tabular}}}}%
    \put(0.67877709,0.01310463){\color[rgb]{0,0,1}\makebox(0,0)[lt]{\lineheight{1.25}\smash{\begin{tabular}[t]{l}$(k_3,l_3)_T$\end{tabular}}}}%
    \put(0,0){\includegraphics[width=\unitlength,page=8]{Relation2.pdf}}%
    \put(0.58092176,0.23062788){\color[rgb]{0,0,1}\makebox(0,0)[lt]{\lineheight{1.25}\smash{\begin{tabular}[t]{l}$(1,0)_T$\end{tabular}}}}%
    \put(0.58302327,0.00623803){\color[rgb]{0,0,1}\makebox(0,0)[lt]{\lineheight{1.25}\smash{\begin{tabular}[t]{l}$(1,-1)_T$\end{tabular}}}}%
  \end{picture}%
\endgroup%